\documentclass[11pt]{article}
\usepackage[a4paper,margin=1in]{geometry}
\usepackage{amsmath,amssymb,amsthm,mathtools}
\usepackage{microtype}
\usepackage[colorlinks=true,linkcolor=blue,citecolor=blue,urlcolor=blue]{hyperref}
\usepackage{authblk}

\usepackage{pgfplots}
\pgfplotsset{compat=1.18}

\newtheorem{theorem}{Theorem}[section]
\newtheorem{proposition}[theorem]{Proposition}
\newtheorem{lemma}[theorem]{Lemma}
\newtheorem{corollary}[theorem]{Corollary}
\newtheorem{remark}[theorem]{Remark}
\newtheorem{definition}[theorem]{Definition}
\newtheorem{assumption}[theorem]{Assumption}

\DeclareMathOperator{\Exp}{Exp}
\newcommand{\E}{\mathbb E}
\newcommand{\Prob}{\mathbb P}
\newcommand{\one}{\mathbf 1}
\newcommand{\dd}{\mathrm d}
\newcommand{\Hb}{\mathsf H_{\mathrm b}}

\title{
Decision-Centric Large Deviations\\ for Data-Driven Capital Buffers in Ruin Models
}

\author[1]{Yf Henkes\thanks{\href{mailto:y.j.henkes@student.tue.nl}{y.j.henkes@student.tue.nl}}}
\author[2]{Bart P.G. van Parys\thanks{\href{mailto:bart.van.parys@cwi.nl}{bart.van.parys@cwi.nl}}}
\author[2,1]{Bert Zwart\thanks{\href{mailto:bert.zwart@cwi.nl}{bert.zwart@cwi.nl}}}
\affil[1]{Eindhoven University of Technology, Eindhoven, The Netherlands}
\affil[2]{CWI, Amsterdam, The Netherlands}

\begin{document}
\maketitle

\begin{abstract}
  We consider an insurance risk model with a random walk structure in which the underlying probability law is unknown.
  A decision maker observes a statistic \(Q_n\) computed from \(n\) historical observations and then needs to choose a capital buffer of the form \(C_n=n f(c,Q_n)\), where \(c=\log(1/\delta)/n\) balances the amount of data and the tolerated ruin probability \(\delta\).
  The classical safe capital buffer is inversely proportional to the \emph{adjustment coefficient} \(\gamma\), the exponential rate at which the ruin probability decays; when the law is unknown, this coefficient must be inferred from the data.
The profile $f$ couples the statistical cost of observing an atypical historical statistic with the future ruin exponent induced by the resulting decision. We study the joint ex ante probability — over both the historical sample and an independent future risk process — that the future maximum exceeds the data-dependent buffer. 
  
  We show that the naive plug-in rule fails to achieve the prescribed logarithmic decay exponent $c$, illustrating the adverse impact of model uncertainty when making decisions under rare-event constraints. We then identify a profile \(f^*=f^*(c,Q_n)\) with the following appealing properties: its lower semicontinuous majorants are safe, while regular continuous rules that fall strictly below \(f^*\) are, under mild conditions, unsafe.
  We illustrate the potential applicability of our framework by developing three parametric examples.
  In nonparametric settings, we show that a single exponential envelope leads to degenerate capital buffers, whereas a two-level exponential envelope yields a nondegenerate capital buffer which we prove to be safe. 
\end{abstract}

\section{Introduction}

Classical ruin theory studies the probability that the surplus process of an
insurance portfolio crosses a prescribed capital buffer level.  This problem is
mathematically equivalent to the all-time maximum \(M\) of a random walk with
increment distribution \(P\) exceeding a level \(x\).  A central result in risk
theory is Lundberg's inequality,
\begin{equation}\label{eq:lundberg-bound}
        \Prob_P\{M>x\}\le \exp\{-\gamma(P)x\},\qquad x\ge0 .
\end{equation}
Here \(\gamma(P)\) is the adjustment coefficient; textbook treatments include
\cite{asmussen2003, asmussen_albrecher, grandell1991}.
If \(\gamma(P)\) is known and the goal is to keep the probability of ruin below
\(\delta\), a safe oracle choice is \(C^*=\log(1/\delta)/\gamma(P)\).
In practice \(P\) is unknown and must be inferred from data.  Replacing
\(\gamma(P)\) by an estimator is not innocuous: the resulting capital buffer level is
random, and the event \(\{M>C_n\}\) can occur through a combination of an
unusually large future maximum and an unusually small data-driven capital
buffer choice.

This paper studies that interaction in the regime
\(\delta=e^{-cn}\), where the amount of data and the logarithmic ruin tolerance
scale together.
We focus on this regime because it is where data uncertainty bites.
Were $\delta$ held fixed, a consistent estimator of the adjustment coefficient would make plug-in calibration asymptotically harmless. When $\log(1/\delta)$ is proportional to $n$, residual statistical fluctuations in the estimated increment law, however, occur on the same exponential scale as the target ruin probability. The decision rule must therefore account for combinations of statistically atypical historical data and an unusually large future maximum.
The central point is that the statistical object of interest is not merely the adjustment coefficient itself, but the rare-event performance of the {\em decision} obtained after inserting data into a capital buffer decision rule.
The safety notion used throughout is \emph{ex ante}: probability is taken jointly over the historical sample and the independent future risk process. Consequently, the large-deviation cost of observing an atypical statistic may contribute to the prescribed exponent \(c\). 

Our work is related to several streams of literature at the interface of ruin
theory, statistics, large deviations, and model uncertainty.  
For work on properties of statistical estimators
for the adjustment coefficient and related quantities see
~\cite{csorgo_teugels1990, embrechts_mikosch1991, grandell1979, hall_teugels_vanmarcke1992, hipp1989, pitts_gruebel_embrechts1996}. 
Large-deviation properties of estimators of adjustment coefficients and related rate
functions were studied in
\cite{duffy_metcalfe2005_cumulative,duffy_metcalfe2005_ld_estimating,duffy_meyn2011}.   For Bayesian viewpoints, see
\cite{ganesh_oconnell1999, ganesh_oconnell2000, macci2014, macci_piccioni2023}.
Thus the existing statistical literature quantifies uncertainty in ruin-related estimators.
In contrast, we do not aim to estimate the true value of the adjustment
coefficient. Instead, we derive data-dependent decisions for which the
joint probability of an atypical historical sample and subsequent ruin
has a prescribed asymptotic decay exponent.

Connections also arise with statistical formulations of robust optimization,
empirical likelihood, and data-driven decision-making; see, for instance,
\cite{agrawal-juneja-glynn-heavy,bental2013,duchi_glynn_namkoong2021,vanparys-zwart-robust-mean}.  The ``data-to-decisions'' viewpoint of
\cite{vanparys_esfahani_kuhn2021} is particularly close in spirit: using
large-deviation ideas, the authors show that certain relative-entropy DRO
formulations are optimal for predictor-prescriptor pairs under out-of-sample
disappointment constraints, and \cite{sutter_vanparys_kuhn2024} develops a
general optimality framework for data-driven optimization. 

Related approaches
to uncertainty quantification and sensitivity analysis for stochastic systems
appear in \cite{atar_chowdhary_dupuis2015,
dupuis_katsoulakis_pantazis_plechac2016,
dupuis_katsoulakis_pantazis_reybellet2020, lam2016}.   Our formulation differs
from standard distributionally robust expected-loss problems: we aim to control the behavior of a rare event at minimal cost, the optimized
quantity is a nonlinear ruin exponent, and the relative-entropy terms enter
through the large-deviation geometry of empirical statistics.

We now give an overview of our contributions and the organization of the rest of this paper. We formulate the problem and state some preliminary results in Section~\ref{sec:problem-formulation}. In
Section~\ref{sec:plugin}, we formally show that the empirical plug-in capital buffer
\(cn/\gamma(P_n)\), with notation formalized below, is unsafe: the exponential
decay rate of the resulting ruin probability is strictly smaller than the
prescribed target decay rate \(c\).  This motivates
a formulation in which the decision maker observes a statistic \(Q_n\) satisfying
an LDP with rate function \(I_P\) under model \(P\), while the unknown law belongs
to a known model class \(\mathcal P\).  For capital buffers of the form
\(C_n=nf(c,Q_n)\), the relevant exponent balances the cost of observing an
atypical statistic against the future ruin exponent induced by the capital buffer
chosen at that statistic:
\begin{equation}
\label{eq:introvariation}
     \inf_q \{ I_P(q)+\gamma(P)f(c,q)\}.
\end{equation}
The goal is to choose \(f(c,q)\) so that this value is at least \(c\) for every
\(P\in\mathcal P\).  This leads to the pointwise profile
\begin{equation}
  \label{keyformula}
  f^*(c,q)=\sup_{P\in\mathcal P}
  \frac{(c-I_P(q))_+}{\gamma(P)}
\end{equation}
which we show represents the tipping point between safe and unsafe profiles.
Indeed, any lower semicontinuous majorant of \(f^*\) is safe, while any regular continuous
rule that falls below \(f^*\) at a point where the LDP lower-bound mechanism can
be localized fails to attain the target exponent; Section~\ref{sec:safe-profiles}
gives the precise statements.

The rest of the paper develops \eqref{keyformula} in two directions.  Section~\ref{sec:examples}
works out three structured parametric examples: Rademacher increments, Gaussian
increments with unknown mean and variance, and differences of independent
exponential variables.  

Section~\ref{sec:nonparametric} treats nonparametric
envelope classes; these come with several conceptual challenges. If one uses only the 
unit exponential envelope \(\mathcal P_a=\{P:\E_Pe^{aX}\le1\}\) for some fixed $a>0$, then the resulting unrestricted
nonparametric problem degenerates to the capital buffer $C_n=nc/a$, since arbitrarily small amounts of unseen
right-tail mass can drive the worst-case adjustment coefficient to the boundary
\(a\).  We therefore introduce a two-level exponential envelope, adding the
constraint \(\E_Pe^{bX}\le B\) for some \(b>a\) and $B>1$.  This stronger envelope rules out
the off-support spike mechanism while preserving a tractable KL-projection
structure. 

A second technical challenge is that the weak-topology version of
Sanov's theorem does not directly control the unbounded exponential
moments entering the nonparametric profile. Weak convergence of
probability laws does not imply convergence of
\(\int e^{r x}\,dQ(x)\), and the regularity conditions needed for a
direct application of Proposition~\ref{prop:rate} or Varadhan's lemma are therefore
not automatic. We instead prove safety directly from the
finite-dimensional dual representation and an exponential aggregation
bound, inspired by similar arguments in \cite{agrawal-juneja-glynn-heavy} and \cite{vanparys-zwart-robust-mean}.

To make the construction implementable, we derive a
finite-dimensional dual representation. Weak duality yields a
conservative profile for arbitrary candidate laws \(Q\). Whenever
\(\mathbb E_Q e^{bX}<\infty\), and hence in particular when \(Q\) is
an empirical law, strong duality holds and the representation is
exact. The resulting profile can be computed through a
finite-dimensional optimization problem with few variables over a
compact feasible set;  we refer to Section~\ref{subsec:two-level-envelope} for details.
We provide a compact numerical illustration of our framework in Section~\ref{sec:numerical-illustration}. Using a Gaussian benchmark distribution, we compare the empirical plug-in rule, the Gaussian parametric rule, and the nonparametric rule in terms of the normalized finite-sample exponent induced by the Cram\'er--Lundberg upper bound. 

\subsection*{Notation}

Throughout, \(X\) denotes a real-valued increment and \(P\) a law for it.  Expectation and probability under \(P\) are written \(\E_P\) and \(\Prob_P\), with the argument in parentheses or set braces; square brackets are reserved for conditional expectations.  The subscript is dropped when the law is clear from context.  We write \(\Lambda_P(\theta)=\log\E_P e^{\theta X}\) for the cumulant generating function, defined wherever the expectation is finite.  For finite nonnegative measures \(\nu,\mu\), the relative entropy is \(D(\nu\|\mu):=\int\log(\dd\nu/\dd\mu)\,\dd\nu\) if \(\nu\ll\mu\) and \(+\infty\) otherwise, where \(\ll\) denotes absolute continuity.  Finally, \(x_+=\max\{x,0\}\), \(\one_A\) is the indicator of a set \(A\), \(\delta_x\) is the Dirac mass at \(x\), and \(P_n=n^{-1}\sum_{i=1}^n\delta_{X_i}\) is the empirical law of a sample \(X_1,\dots,X_n\).

\section{Problem formulation}
\label{sec:problem-formulation}

Let \(\mathcal P\) be a class of probability laws for real-valued increments
\(X\), and let \(\mathcal Q\) be the topological space in which the observed statistic takes its values; in the examples below \(\mathcal Q\) is Polish. 
We impose throughout the standard ruin-theoretic assumptions on the increment law
\(P\in\mathcal P\), and pause to record the role of each. First, we assume
\(\mathbb E_P|X|<\infty\), so that the drift \(\mathbb E_P X\) is well defined and the
cumulant generating function has right derivative
\(\Lambda_P'(0+)=\mathbb E_P X\). Second, we impose the negative-drift condition
\(\mathbb E_P X<0\), which ensures that the all-time maximum is finite and ruin is
not certain. Third, \(\Prob_P(X>0)>0\) is a nontriviality condition, ensuring in
particular that \(\Lambda_P\) can eventually become positive. Finally, we assume
that \(\Lambda_P\) is finite on some interval \([0,b]\), with \(b>0\), or
equivalently that \(\mathbb E_P e^{bX}<\infty\). On \([0,b]\), the function
\(\Lambda_P\) is convex, with \(\Lambda_P(0)=0\) and
\(\Lambda_P'(0+)=\mathbb E_P X<0\), so it initially drops below \(0\). The remainder of the assumptions pertain to our ruin and observation model:

\paragraph{Ruin model.} For an arbitrary increment law we define the adjustment coefficient by
\[
        \gamma(P):=\sup\{\theta\ge0:\Lambda_P(\theta)\le0\}=\sup\{\theta\ge0:\E_P e^{\theta X}\le1\},
\]
which is nonnegative because \(\theta=0\) is always feasible (\(\Lambda_P(0)=0\)).  Since \(\Lambda_P\) is convex and lower semicontinuous with \(\Lambda_P(0)=0\), the set \(\{\theta\ge0:\Lambda_P(\theta)\le0\}\) is a closed interval \([0,\gamma(P)]\), so \(\gamma(P)\ge\vartheta\iff\E_P e^{\vartheta X}\le1\) for every \(\vartheta>0\).  The Lundberg upper bound \eqref{eq:lundberg-bound} holds at this generality: the supermartingale argument bounds \(\Prob_P\{M>x\}\le e^{-\theta x}\) whenever \(\E_P e^{\theta X}\le1\), hence \(\Prob_P\{M>x\}\le e^{-\gamma(P)x}\).  As a standing assumption we require the true increment law to satisfy the stronger Cram\'er condition \(\E_P e^{\gamma X}=1\) with \(\gamma=\gamma(P)\in(0,b)\); by convexity this positive root is then unique and equals \(\gamma(P)\).  Under it the coefficient governs the decay of the future ruin probability exactly: writing \(M\) for the all-time maximum of an independent future random walk with increment law \(P\), the following asymptotic, due to Cram\'er and Lundberg, holds \cite{asmussen2003} alongside the upper bound \eqref{eq:lundberg-bound} (see also \cite[Proposition~2.1]{nuyens_zwart2005}):
\begin{equation}\label{eq:ruin-log-asymp}
  \lim_{x\to\infty}\frac1x\log \Prob_P\{M>x\}=-\gamma(P).
\end{equation}
We record two analytic facts about the adjustment coefficient.  Because \(\gamma\) is interior to \([0,b]\), \cite[Lemma~2.2.5]{dembo_zeitouni} shows that \(\Lambda_P\) is differentiable at \(\gamma\), with \(\Lambda_P'(\gamma)=\E_P Xe^{\gamma X}/\E_P e^{\gamma X}\); since \(\E_P e^{\gamma X}=1\), this equals \(m_\gamma:=\E_P Xe^{\gamma X}\).  As \(\Lambda_P\) is convex with \(\Lambda_P(0)=\Lambda_P(\gamma)=0\) and \(\Lambda_P'(0^+)<0\), its slope at the second root \(\gamma\) is strictly positive; hence \(m_\gamma\in(0,\infty)\).  The interior condition also licenses differentiating \(\Lambda_P\) under the expectation near \(\gamma\); this is what drives the perturbation argument of Section~\ref{sec:plugin}, where we show that plug-in calibration of \(\gamma\) is unsafe (Proposition~\ref{prop:plugin}).

\paragraph{Observation model.} The statistic observed by the decision maker is a \(\mathcal Q\)-valued random
variable \(Q_n\) computed from \(n\) i.i.d.\ historical observations
\(X_1,\ldots,X_n\) which are independent of $M$, with empirical law \(P_n\).
In a nonparametric setting, we may take \(Q_n=P_n\), but \(Q_n\) may also be a simple statistic such as a sample mean
and variance.  A data-driven capital buffer is a measurable function of the form \(C_n=n f(c,Q_n)\), where \(c>0\) is the target exponent connecting the ruin probability tolerance $\delta$ with the amount of data $n$ through $\delta=e^{-cn}$.
We assume that the statistic $Q_n$ satisfies a large-deviation principle.  We use 
standard terminology and notation from large-deviation theory
\cite{dembo_zeitouni,dupuis_ellis}.

\begin{assumption}[LDP for the data statistic]\label{ass:LDP}
For every \(P\in\mathcal P\), the random variable \(Q_n\) satisfies an LDP on
\(\mathcal Q\) with speed \(n\) and good rate function
\(I_P:\mathcal Q\to[0,\infty]\).  Thus, for closed \(F\subset\mathcal Q\),
\[
\limsup_{n\to\infty}\frac1n\log \Prob_P\{Q_n\in F\}
\le -\inf_{q\in F}I_P(q),
\]
and for open \(G\subset\mathcal Q\),
\[
\liminf_{n\to\infty}\frac1n\log \Prob_P\{Q_n\in G\}
\ge -\inf_{q\in G}I_P(q).
\]
\end{assumption}

\paragraph{Objective.} We say that a measurable function \(f: (0,\infty) \times {\mathcal Q}\rightarrow [0,\infty] \) is \emph{safe at
exponent \(c\)} for \(\mathcal P\) if
\begin{equation}
\label{eq:safe}
     \limsup_{n\to\infty}\frac1n\log
        \Prob_P\{M>n f(c,Q_n)\}
        \le -c,
        \qquad P\in\mathcal P.
\end{equation}

The safety criterion in \eqref{eq:safe} is an \emph{ex ante} criterion: the probability is taken jointly over the historical observations, through \(Q_n\), and the independent future maximum \(M\).  Under the present criterion, the large-deviation cost of observing an atypical statistic may contribute to the overall decay exponent of the joint event. Thus, a statistically unlikely observation can partially or entirely account for the target exponent \(c\).

In some cases, no finite $f$ exists given a fixed exponent $c$, which is the reason we allow $f$ to take infinite values. In the explicit parametric examples presented in Section~\ref{sec:examples}, we will present intuition on when this can happen.

Under the assumptions of the ruin model and of the observation model, the two independent sources of randomness (the future maximum \(M\), controlled through the Lundberg bound \eqref{eq:lundberg-bound} and the asymptotic \eqref{eq:ruin-log-asymp}, and the statistic \(Q_n\), obeying the LDP of Assumption~\ref{ass:LDP}) combine into a single coupled setting.  A standard large-deviation argument then governs the data-driven ruin event \(\{M>nf(c,Q_n)\}\) and identifies its exponential decay rate.

\begin{proposition}[Decay rate of a data-driven ruin probability]\label{prop:rate}
Let Assumption~\ref{ass:LDP} hold and let
\(f(c,\cdot):\mathcal Q\to[0,\infty]\) be measurable.  Define
\[
        J_P(f,c):=
        \inf_{q\in\mathcal Q}\{I_P(q)+\gamma(P)f(c,q)\}.
\]
If \(q\mapsto f(c,q)\) is lower semicontinuous, then the Lundberg bound
\eqref{eq:lundberg-bound} implies
\[
\limsup_{n\to\infty}\frac1n\log
\Prob_P\{M>n f(c,Q_n)\}
\le -J_P(f,c).
\]
If, in addition, the logarithmic ruin asymptotic \eqref{eq:ruin-log-asymp}
holds under \(P\), define
\[
        \mathcal C_f:=
        \{q\in\mathcal Q:f(c,q)<\infty
        \text{ and }f(c,\cdot)\text{ is continuous at }q\}
\]
and
\[
        J_P^{\rm cont}(f,c):=
        \inf_{q\in\mathcal C_f}
        \{I_P(q)+\gamma(P)f(c,q)\},
\]
with the convention that the infimum over the empty set is \(+\infty\).  Then
\[
\liminf_{n\to\infty}\frac1n\log
\Prob_P\{M>n f(c,Q_n)\}
\ge -J_P^{\rm cont}(f,c).
\]
\end{proposition}

\begin{proof}
By independence of $Q_n$ and $M$, and (\ref{eq:lundberg-bound}),
\[
\Prob_P\{M>n f(c,Q_n)\}
\le
\E_P\exp\{-n\gamma(P)f(c,Q_n)\}.
\]
Set \(\psi(q)=-\gamma(P)f(c,q)\).  If \(f(c,\cdot)\) is lower semicontinuous,
then \(\psi\) is upper semicontinuous and \(\psi\le0\).  For \(K>0\), put
\(\psi_K=\max\{\psi,-K\}\). 
The Laplace upper bound for bounded upper semicontinuous functions, which is
the upper-bound part of Varadhan's lemma (see, for instance~\cite[Lemma~4.3.6]{dembo_zeitouni}), gives
\[
\limsup_{n\to\infty}\frac1n\log \E_Pe^{n\psi_K(Q_n)}
\le \sup_{q\in\mathcal Q}\{\psi_K(q)-I_P(q)\}.
\]
Writing \(S=\sup_q\{\psi(q)-I_P(q)\}=-J_P(f,c)\), we have \(\sup_q\{\psi_K(q)-I_P(q)\}\le \max\{S,-K\}\).  Letting \(K\to\infty\) proves the upper bound.

For the lower bound, fix \(q_0\in\mathcal C_f\) and \(\varepsilon>0\).  By
continuity there is an open neighbourhood \(G\) of \(q_0\) such that \(f(c,q)\le f(c,q_0)+\varepsilon\) for all \(q\in G\).  On \(\{Q_n\in G\}\), the event
\(\{M>n(f(c,q_0)+\varepsilon)\}\) implies
\(\{M>n f(c,Q_n)\}\).  Hence, by independence,
\[
\Prob_P\{M>n f(c,Q_n)\}
\ge
\Prob_P\{Q_n\in G\}
\Prob_P\{M>n(f(c,q_0)+\varepsilon)\}.
\]
The LDP lower bound and \eqref{eq:ruin-log-asymp} yield
\[
\liminf_{n\to\infty}\frac1n\log
\Prob_P\{M>n f(c,Q_n)\}
\ge
-\inf_{q\in G}I_P(q)-\gamma(P)(f(c,q_0)+\varepsilon).
\]
Since \(q_0\in G\), the right-hand side is at least
\(-I_P(q_0)-\gamma(P)f(c,q_0)-\gamma(P)\varepsilon\).  Let
\(\varepsilon\downarrow0\), and then take the supremum over
\(q_0\in\mathcal C_f\).  This proves the lower bound.
\end{proof}

The two bounds of Proposition~\ref{prop:rate} match whenever the variational values coincide, \(J_P^{\rm cont}(f,c)=J_P(f,c)\).  In that case the data-driven ruin probability obeys the sharp logarithmic asymptotic
\[
        \Prob_P\{M>n f(c,Q_n)\}
        =\exp\{-nJ_P(f,c)+o(n)\}.
\]
The two values can differ only through the boundary and discontinuity effects that the restriction to \(\mathcal C_f\) excludes; for the continuous, finite-valued profiles used below they agree, so the decay rate is exactly \(J_P(f,c)\).

\section{Plug-in estimators of the adjustment coefficient are unsafe}\label{sec:plugin}

We next show why direct empirical substitution of the adjustment coefficient is
not a reliable way to choose a capital buffer.  Let \(P\) be the true increment law, satisfying the standing assumptions of Section~\ref{sec:problem-formulation}; in particular \(m_\gamma:=\E_P Xe^{\gamma X}\in(0,\infty)\) and \(\Lambda_P\) is finite on \([0,b]\) with \(\gamma=\gamma(P)\in(0,b)\).
The naive plug-in capital buffer targeting exponent \(c>0\) is \(C_n^{\rm plug}:=cn/\gamma(P_n)\), where
\[
        \gamma(P_n)=\sup\left\{s\ge0:\frac1n\sum_{i=1}^n e^{sX_i}\le1\right\},
\]
the adjustment coefficient of Section~\ref{sec:problem-formulation} evaluated at the empirical law \(P_n\). We set \(C_n^{\rm plug}=+\infty\) whenever \(\gamma(P_n)\notin(0,\infty)\).

The following proposition formalizes the failure of this direct empirical substitution. 
It can be viewed as a large-deviation analog of the optimism of plug-in empirical decision rules \cite{michaud1989markowitz}.

\begin{proposition}[Plug-in calibration misses the target exponent]\label{prop:plugin}
Under the assumptions above, for every \(c>0\),
\[
        \limsup_{n\to\infty}
        -\frac1n\log \Prob_P\{M>C_n^{\rm plug}\}<c.
\]
Thus the plug-in rule fails to attain the target exponent \(c\).
\end{proposition}

Before giving the proof, we record the heuristic exponent calculation behind the result. If the empirical law \(P_n\) is close to a law \(Q\), then the plug-in rule uses the capital buffer \(cn/\gamma(Q)\). Under the true law \(P\), the future ruin probability at this capital buffer has logarithmic order \(\exp\{-n c\gamma(P)/\gamma(Q)\}\).  By Sanov's theorem, the empirical fluctuation \(P_n\approx Q\) has probability of order \(\exp\{-n D(Q\|P)\}\), where \(D(Q\|P)\) is the relative entropy of \(Q\) with respect to \(P\). Thus the plug-in rule is expected to lead to a data-driven ruin probability with exponential decay rate
\begin{equation}\label{eq:plugin-heuristic}
     \inf_{Q}\left\{
        D(Q\|P)+c\frac{\gamma(P)}{\gamma(Q)}
        \right\}.
\end{equation}
At \(Q=P\) this expression equals \(c\). However, under a local likelihood-ratio perturbation \(\dd Q_\varepsilon/\dd P=1+\varepsilon h\), the change in \(\gamma(Q_\varepsilon)\) can be linear in \(\varepsilon\), whereas \(D(Q_\varepsilon\|P)=O(\varepsilon^2)\). A perturbation that increases the adjustment coefficient therefore lowers the second term in \eqref{eq:plugin-heuristic} at first order in \(\varepsilon\), while the relative entropy \(D(Q_\varepsilon\|P)\) changes only in the second order. This causes the value in \eqref{eq:plugin-heuristic} to drop below \(c\). The proof makes this mechanism rigorous by constructing such a perturbation explicitly.

\begin{proof}[Proof of Proposition \ref{prop:plugin}]
The drift assumption \(\E_PX<0\) forces \(
\Prob_P(X<0)>0\), while \(\Prob_P(X>0)>0\) is the nontriviality assumption; both half-lines thus carry positive mass.  Choose measurable sets \(A_-\subset(-\infty,0)\) and \(A_+\subset(0,\infty)\) with \(P(A_-)>0\) and \(P(A_+)>0\), and define \(h(x)=\one_{A_-}(x)/P(A_-)-\one_{A_+}(x)/P(A_+)\).  Then \(h\) is bounded and \(\E_Ph=0\).  Moreover,
\[
        \E_P h(X)e^{\gamma X}
        =\E_P[e^{\gamma X}\mid X\in A_-]
        -\E_P[e^{\gamma X}\mid X\in A_+]<0,
\]
because \(\gamma>0\) makes \(e^{\gamma X}<1\) on \(A_-\subset(-\infty,0)\) and \(e^{\gamma X}>1\) on \(A_+\subset(0,\infty)\); the two conditional expectations are finite since \(\E_P e^{\gamma X}=1\) by the Cram\'er condition.  Put \(\beta_h:=-\E_P h(X)e^{\gamma X}\in(0,\infty)\).  Since \(h\) is bounded, \(1+\varepsilon h>0\) for all sufficiently small \(\varepsilon>0\); define \(\dd Q_\varepsilon/\dd P=1+\varepsilon h\).  Because \(\E_P h=0\) the first-order term vanishes, and boundedness of \(h\) makes the cubic remainder uniform, so
\[
        D(Q_\varepsilon\|P)
        =\E_P\{(1+\varepsilon h)\log(1+\varepsilon h)\}
        =\frac{\varepsilon^2}{2}\E_Ph^2+O(\varepsilon^3).
\]
In particular, \(D(Q_\varepsilon\|P)=O(\varepsilon^2)\).

Since \(m_\gamma\in(0,\infty)\), the interval \((0,\beta_h/m_\gamma)\) is nonempty; fix \(r\) in it and set \(t_\varepsilon=\gamma+r\varepsilon\).  For small \(\varepsilon>0\) we have \(t_\varepsilon\le t_0\) for some fixed \(t_0<b\).  On \([\gamma,t_0]\) the integrand admits a fixed \(P\)-integrable dominating function: on \(\{X>0\}\), \(|Xe^{tX}|\le Ce^{bX}\) with \(C:=\sup_{x>0}xe^{-(b-t_0)x}<\infty\), while on \(\{X\le0\}\) and \(t\ge\gamma\), \(|Xe^{tX}|\le\sup_{x\le0}|x|e^{\gamma x}=(e\gamma)^{-1}\).  Since \(h\) is bounded, this licenses differentiating under the expectation and yields
\[
\begin{aligned}
        \E_{Q_\varepsilon}e^{t_\varepsilon X}
        &=\E_P\{(1+\varepsilon h)e^{(\gamma+r\varepsilon)X}\}\\
        &=1+\varepsilon\{\E_P he^{\gamma X}
          +r\E_P Xe^{\gamma X}\}+o(\varepsilon)\\
        &=1-\varepsilon(\beta_h-rm_\gamma)+o(\varepsilon),
\end{aligned}
\]
where \(\E_P he^{\gamma X}=-\beta_h\) and \(\E_P Xe^{\gamma X}=m_\gamma\).  Since \(r<\beta_h/m_\gamma\) gives \(\beta_h-rm_\gamma>0\), for all sufficiently small \(\varepsilon>0\) we have \(\E_{Q_\varepsilon}e^{t_\varepsilon X}<1\).  Choose \(\chi_\varepsilon>0\) such that
\(\E_{Q_\varepsilon}e^{t_\varepsilon X}<1-\chi_\varepsilon\), and let
\[
        E_n^\varepsilon=
        \left\{
        \frac1n\sum_{i=1}^n e^{t_\varepsilon X_i}
        \le1-\frac{\chi_\varepsilon}{2},\quad
        \frac1n\sum_{i=1}^n\one_{A_+}(X_i)\ge\frac{Q_\varepsilon(A_+)}2
        \right\}.
\]
Write \(P^n\) and \(Q_\varepsilon^n\) for the product measures governing the
historical sample under \(P\) and \(Q_\varepsilon\), respectively.  By the law of
large numbers under \(Q_\varepsilon\), \(Q_\varepsilon^n(E_n^\varepsilon)\to1\).
On \(E_n^\varepsilon\), the empirical law has positive mass on \((0,\infty)\)
and satisfies
\(\E_{P_n}e^{t_\varepsilon X}<1\).  Hence \(\gamma(P_n)\) is finite and
\(\gamma(P_n)\ge t_\varepsilon\).  Consequently, \(C_n^{\rm plug}\le cn/t_\varepsilon\) on \(E_n^\varepsilon\).

It remains to estimate the probability of \(E_n^\varepsilon\) under the true
law.  Let \(L_n=n^{-1}\sum_{i=1}^n\log(1+\varepsilon h(X_i))\).  Under \(Q_\varepsilon\), \(L_n\to \E_{Q_\varepsilon}\log(1+\varepsilon h)=D(Q_\varepsilon\|P)\) almost surely.  Thus, for every \(\eta>0\),
\[
        \Prob_{Q_\varepsilon^n}
        \{E_n^\varepsilon,
        L_n\le D(Q_\varepsilon\|P)+\eta\}\to1.
\]
We now transfer this typical event under $Q_\varepsilon$ back to $P$.
Since \(\dd P^n/\dd Q_\varepsilon^n=\exp\{-nL_n\}\), it follows that
\[
        \limsup_{n\to\infty}
        -\frac1n\log \Prob_{P^n}\{E_n^\varepsilon\}
        \le D(Q_\varepsilon\|P)+\eta.
\]
Letting \(\eta\downarrow0\) gives
\[
        \limsup_{n\to\infty}
        -\frac1n\log \Prob_{P^n}\{E_n^\varepsilon\}
        \le D(Q_\varepsilon\|P).
\]
By independence of the data and the future maximum,
\[
        \Prob_P\{M>C_n^{\rm plug}\}
        \ge
    \Prob_{P^n}\{E_n^\varepsilon\}
        \Prob_P\{M>\frac{cn}{t_\varepsilon}\}.
\]
Using the logarithmic ruin asymptotic \eqref{eq:ruin-log-asymp}, which holds under \(P\) by the standing assumptions,
\[
        \limsup_{n\to\infty}
        -\frac1n\log \Prob_P\{M>C_n^{\rm plug}\}
        \le
        D(Q_\varepsilon\|P)+c\frac{\gamma}{t_\varepsilon}.
\]
Finally, \(c\gamma/t_\varepsilon=c\gamma/(\gamma+r\varepsilon)=c-(cr/\gamma)\varepsilon+O(\varepsilon^2)\), whereas \(D(Q_\varepsilon\|P)=O(\varepsilon^2)\).  The right-hand side is
strictly smaller than \(c\) for all sufficiently small positive
\(\varepsilon\).  This proves the claim.
\end{proof}

This failure is visible in numerical experiments.  Section~\ref{sec:numerical-illustration} reports that, for the benchmark law \(P=N(-1,1)\) at the moderate sample size \(n=50\), the plug-in rule attains an empirical safety level strictly below the target exponent \(c\), consistent with Proposition~\ref{prop:plugin}, whereas the safe profiles developed in the sections that follow stay at or above it.

\section{Pointwise lower bounds on safe profiles}\label{sec:safe-profiles}

The preceding result motivates a rule that is designed directly for the
large-deviation exponent of the resulting decision.  Proposition~\ref{prop:rate}
shows that a lower semicontinuous profile \(f\) is safe whenever \(I_P(q)+\gamma(P)f(c,q)\ge c\) for all \(q\in\mathcal Q\) and \(P\in\mathcal P\).  This pointwise family of inequalities leads to the following envelope $f^*$ which we will refer to as the pointwise minimal safe profile.

\begin{definition}[Pointwise minimal safe profile]\label{def:fstar}
For \(c>0\), define
\[
        f^*(c,q)
        =
        \sup_{P\in\mathcal P}
        \frac{(c-I_P(q))_+}{\gamma(P)}
        =
        \sup_{\substack{P\in\mathcal P\\ I_P(q)<c}}
        \frac{c-I_P(q)}{\gamma(P)},
\]
with the convention that the supremum may be \(+\infty\).  If the set in the
second supremum is empty, the value is zero.
\end{definition}

No regularity of \(f^*\) is automatic from the definition, and in general it cannot be taken for granted.  Safety must therefore be established case by case.  When \(f^*\) is lower semicontinuous on the region of interest one may work with it directly, as in the parametric examples of Section~\ref{sec:examples}, or pass to a lower semicontinuous majorant through Theorem~\ref{thm:safety}.  When regularity is out of reach, safety is instead proved by a direct argument that bypasses it altogether, as in the nonparametric envelope classes of Section~\ref{sec:nonparametric}.

\begin{theorem}[Safety of lower semicontinuous majorants]\label{thm:safety}
Assume Assumption~\ref{ass:LDP}.  Let \(g(c,\cdot):\mathcal Q\to[0,\infty]\)
be lower semicontinuous and satisfy \(g(c,q)\ge f^*(c,q)\) for all \(q\in\mathcal Q\).  Then \(C_n=n g(c,Q_n)\) is safe at exponent \(c\):
\[
        \limsup_{n\to\infty}\frac1n\log
        \Prob_P\{M>n g(c,Q_n)\}
        \le -c,
        \qquad P\in\mathcal P.
\]
\end{theorem}

\begin{proof}
For every \(P\in\mathcal P\) and \(q\in\mathcal Q\), \(g(c,q)\ge f^*(c,q)\ge(c-I_P(q))_+/\gamma(P)\).  Hence \(I_P(q)+\gamma(P)g(c,q)\ge c\) for all \(q\).  Proposition~\ref{prop:rate}
and  Lundberg's inequality give the result.
\end{proof}

We next give the corresponding local necessity statement.  It is formulated in
terms of a generic regular rule rather than a particular representative of the
lower semicontinuous envelope.  This restriction is intrinsic to the method rather than a technical convenience.  Large-deviation lower bounds are one-sided and ``rough'': they control \(\Prob_P(Q_n\in G)\) only for \emph{open} sets \(G\).  Any impossibility result derived directly from an LDP must therefore localize the statistic in an open neighbourhood and read off the candidate profile there, which pins its value down only when the profile is regular enough to be controlled on that neighbourhood by its limit.  Some regularity condition is thus unavoidable for LDP-based lower bounds, and we make one explicit next.

\begin{definition}[Regular asymptotically linear rules]\label{def:regular}
A sequence of capital buffers \(C_n'=F_n(c,Q_n)\) is regular with limiting profile
\(f'(c,\cdot)\) if \(n^{-1}F_n(c,q)\to f'(c,q)\) uniformly on a neighbourhood of each point of \(\mathcal Q\), and \(f'(c,\cdot)\) is continuous.
\end{definition}

In our next result, we show that every regular safe rule must lie above \(f^*\) at all points where
the LDP lower-bound mechanism can be localized. Note that $f^*$ itself may not be regular, though we find it to be in several examples below.

\begin{theorem}[Pointwise lower bound for regular rules]\label{thm:pointwise-lower}
Assume Assumption~\ref{ass:LDP}, and let
$C_n'=F_n(c,Q_n)$
be regular with limiting profile \(f'(c,\cdot)\). If there exist
\(P_0\in\mathcal P\) and \(q_0\in\mathcal Q\) such that the logarithmic ruin asymptotic \eqref{eq:ruin-log-asymp} holds under \(P_0\),
\[
I_{P_0}(q_0)<\infty \mbox{ and }
I_{P_0}(q_0)
+
\gamma(P_0)f'(c,q_0)
<
c,
\]
then \(C_n'\) is not safe at exponent \(c\) under \(P_0\):
\[
\liminf_{n\to\infty}
\frac1n
\log
\mathbb P_{P_0}\{M>C_n'\}
>
-c.
\]
\end{theorem}

\begin{proof}
Write \(\gamma_0=\gamma(P_0)\). Choose \(\eta>0\) such that
$I_{P_0}(q_0)+\gamma_0f'(c,q_0)<c-4\eta$.
By regularity, there is a neighbourhood of \(q_0\) on which
\(n^{-1}F_n(c,\cdot)\) converges uniformly to \(f'(c,\cdot)\).
Using the continuity of \(f'(c,\cdot)\), choose an open neighbourhood
\(U\) of \(q_0\), contained in that neighbourhood, such that
\[
\sup_{q\in U}f'(c,q)
<
f'(c,q_0)+\frac{\eta}{\gamma_0}.
\]
Since \(q_0\in U\),
$\inf_{q\in U}I_{P_0}(q)
\leq
I_{P_0}(q_0)$,
and consequently
\[
\inf_{q\in U}I_{P_0}(q)
+
\gamma_0\sup_{q\in U}f'(c,q)
<
c-3\eta.
\]
By local uniform convergence, for all sufficiently large \(n\),
\[
\sup_{q\in U}\frac1nF_n(c,q)
\leq
\sup_{q\in U}f'(c,q)
+
\frac{\eta}{\gamma_0}.
\]
Therefore, on the event \(\{Q_n\in U\}\),
\[
C_n'
\leq
n\left(
    \sup_{q\in U}f'(c,q)
    +
    \frac{\eta}{\gamma_0}
\right).
\]
Independence of the historical sample and the future maximum gives
\[
\mathbb P_{P_0}\{M>C_n'\}
\geq
\mathbb P_{P_0}\{Q_n\in U\}
\mathbb P_{P_0}\left\{
M>
n\left(
    \sup_{q\in U}f'(c,q)
    +
    \frac{\eta}{\gamma_0}
\right)
\right\}.
\]
The LDP lower bound and the logarithmic ruin asymptotic now yield
\[
\begin{aligned}
\liminf_{n\to\infty}
\frac1n\log\mathbb P_{P_0}\{M>C_n'\}
&\geq
-\inf_{q\in U}I_{P_0}(q)
-\gamma_0\sup_{q\in U}f'(c,q)
-\eta>-c.
\end{aligned}
\]
\end{proof}

\begin{corollary}[Pointwise lower bound where the supremum is witnessed]
Assume Assumption~\ref{ass:LDP}, and let \(C_n'=F_n(c,Q_n)\) be regular with limiting profile
\(f'(c,\cdot)\). Suppose that there exist \(q_0\in\mathcal Q\) and
\(P_0\in\mathcal P\) such that \eqref{eq:ruin-log-asymp} holds under \(P_0\),
$I_{P_0}(q_0)<\infty$,
and
\[
f'(c,q_0)
<
\frac{c-I_{P_0}(q_0)}{\gamma(P_0)}.
\]
Then \(C_n'\) is not safe at exponent \(c\). In particular, no regular
rule can lie strictly below \(f^*\) at a point where the strict
inequality is witnessed by such a law \(P_0\).
\end{corollary}

\begin{proof}
The displayed inequality is equivalent to
$I_{P_0}(q_0)+\gamma(P_0)f'(c,q_0)<c$,
so the assertion follows from Theorem~\ref{thm:pointwise-lower}.
\end{proof}

The value \(f^*(c,q)\) may be infinite.  This means that the observed statistic
has not ruled out laws with arbitrarily small adjustment coefficient at
information cost below \(c\), so no capital buffer of order \(n\) can achieve
the target exponent uniformly over \(\mathcal P\).  Define
\[
        c_{\rm crit}(q)=\lim_{\varepsilon\downarrow0}
        \inf\{I_P(q):P\in\mathcal P,
        0<\gamma(P)\le\varepsilon\},
\]
where the infimum over the empty set is \(+\infty\).  The limit exists in
\([0,+\infty]\), since the displayed infimum is monotone as
\(\varepsilon\downarrow0\).
The quantity $c_{\rm crit}(q)$ may be interpreted as the information cost of making the adjustment coefficient arbitrarily small while still explaining $q$.
Once $c$ exceeds this value under the observation $Q_n\approx q$, the only way to be safe is to prescribe infinite capital.

\begin{proposition}[Finite/infinite phase transition]\label{prop:phase}
If \(c>c_{\rm crit}(q)\), then \(f^*(c,q)=+\infty\).  If \(c<c_{\rm crit}(q)\), then
\(f^*(c,q)<\infty\).
\end{proposition}

\begin{proof}
If \(c>c_{\rm crit}(q)\), choose \(\eta>0\) such that \(c_{\rm crit}(q)<c-\eta\).  By the definition
of \(c_{\rm crit}(q)\), there are laws \(P_k\) with \(\gamma(P_k)\downarrow0\) and
\(I_{P_k}(q)<c-\eta\).  Hence \((c-I_{P_k}(q))/\gamma(P_k)\to\infty\), so \(f^*(c,q)=+\infty\).

If \(c<c_{\rm crit}(q)\), then for some \(\varepsilon>0\), \(\inf\{I_P(q):0<\gamma(P)\le\varepsilon\}>c\).  Thus every law contributing a positive numerator in the definition of \(f^*\)
has \(\gamma(P)>\varepsilon\).  Therefore \(f^*(c,q)\le c/\varepsilon<\infty\).
\end{proof}

No general conclusion is asserted at the boundary
\(c=c_{\mathrm{crit}}(q)\). At equality, the value of \(f^*(c,q)\) may
be finite or infinite, depending on how rapidly \(I_P(q)\) approaches
\(c_{\mathrm{crit}}(q)\) along laws for which
\(\gamma(P)\downarrow0\). The boundary value is determined explicitly
in each of the parametric examples below.

\section{Parametric examples}\label{sec:examples}

The following three examples illustrate the general rule in structured parametric model uncertainty classes.
Each example has an explicit rate function and an explicit adjustment coefficient, so the
abstract profile can be reduced to a closed-form expression.  The proofs are
collected in Appendix~\ref{app:examples}.

We shall repeatedly use the following equivalent form of Definition~\ref{def:fstar}.  For fixed
\(q\), define
\begin{equation}
     F_q(t):=\inf_{P\in\mathcal P}
        \{I_P(q)+\gamma(P)t\},\qquad t\ge0.
\end{equation}
Then the following simple but useful representation holds:
\begin{lemma}
\label{lem:threshold-representation}
   \begin{equation}\label{eq:threshold-representation}
        f^*(c,q)=\inf\{t\ge0:F_q(t)\ge c\},
\end{equation}
with the convention that the infimum over the empty set is \(+\infty\).
Moreover, \(F_q\) is nondecreasing with \(\lim_{t\to\infty}F_q(t)=c_{\rm crit}(q)\).
\end{lemma}

 \begin{proof}
     Observe that
\(F_q(t)\ge c\) is equivalent to
\(t\ge(c-I_P(q))_+/\gamma(P)\) for every \(P\in\mathcal P\), which proves \eqref{eq:threshold-representation}.

For the second claim, each map \(t\mapsto I_P(q)+\gamma(P)t\) is nondecreasing, hence so is \(F_q\).  For every \(\varepsilon>0\), splitting the infimum defining \(F_q\) according to \(\gamma(P)\le\varepsilon\) or \(\gamma(P)>\varepsilon\) gives
\[
        F_q(t)\ \ge\ \min\big\{\inf\{I_P(q):0<\gamma(P)\le\varepsilon\},\ \varepsilon t\big\};
\]
letting \(t\to\infty\) and then \(\varepsilon\downarrow0\) yields \(\lim_{t\to\infty}F_q(t)\ge c_{\rm crit}(q)\).  Conversely, restricting the infimum to laws with \(0<\gamma(P)\le\varepsilon\) gives, for every \(t\ge0\),
\(F_q(t)\le\inf\{I_P(q):0<\gamma(P)\le\varepsilon\}+\varepsilon t\); letting \(\varepsilon\downarrow0\) shows \(F_q(t)\le c_{\rm crit}(q)\) for every \(t\).
 \end{proof}

 Each profile \(f^*(c,\cdot)\) displayed below solves the threshold representation \eqref{eq:threshold-representation} in closed form.  In every case the finite branch is continuous and the profile jumps to \(+\infty\) across a boundary; this jump is exactly the finite/infinite phase transition of Proposition~\ref{prop:phase}, and the critical exponent \(c_{\rm crit}\) marking it is the relative-entropy cost of driving the fitted model to the zero-adjustment boundary---\(p\to1/2\) for Rademacher increments, \(\mu\to0\) for the Gaussian, and \(\mu\downarrow\lambda\) for the exponential difference---identified explicitly below.  Because \(c_{\rm crit}\) is continuous, the set \(\{q:f^*(c,q)=+\infty\}=\{q:c>c_{\rm crit}(q)\}\) is open, so \(f^*(c,\cdot)\) is lower semicontinuous on all of \(\mathcal Q\).  By Theorem~\ref{thm:safety} the profile \(f^*\) is then itself safe: in these examples no lower semicontinuous majorant is needed.

An important interpretative point concerns the zero-buffer regimes
that occur when the observed statistic lies outside the fitted stable
region. Because our safety criterion is ex ante, the prescription
\(f^*(c,q)=0\) does not mean that future ruin risk is absent conditional
on the observed statistic. Rather, it means that every stable model
capable of explaining the observation incurs an information cost of at
least \(c\), so that the statistical rarity of the data already supplies
the target exponent in the joint data-and-ruin probability.

\subsection{Rademacher increments}\label{subsec:rademacher}

Let \(X\in\{-1,+1\}\) with \(\Prob(X=+1)=p\), where \(0<p<1/2\), so that \(\E X=2p-1<0\); the model class is \(\mathcal P=\{p:0<p<1/2\}\).  We observe the empirical type \(Q_n=n^{-1}\sum_{i=1}^n \one_{\{X_i=+1\}}\), which takes values in \(\mathcal Q=[0,1]\), and we write \(q\in[0,1]\) for a generic value.  The adjustment coefficient and the LDP rate function (Cram\'er's theorem for Bernoulli variables) are
\[
        \gamma(p)=\log\frac{1-p}{p},
        \qquad
        I_p(q)=d(q\|p)
        =q\log\frac{q}{p}+(1-q)\log\frac{1-q}{1-p},
\]
with the usual conventions at \(q\in\{0,1\}\).  Let \(\Hb(u)=-u\log u-(1-u)\log(1-u)\) denote the binary entropy.

We call \(\mathcal Q_s = [0,1/2) \subseteq \mathcal Q\) the \emph{stable region}; on its complement \(\mathcal Q\setminus \mathcal Q_s = [1/2,1]\) the empirical type already indicates nonnegative drift.  The following proposition gives \(f^*\) on all of \(\mathcal Q=[0,1]\).

\begin{proposition}[Explicit Rademacher profile]\label{prop:rademacher}
For every \(q\in[0,1]\), the critical exponent of Proposition~\ref{prop:phase} is
\[
  c_{\rm crit}(q)=d(q\|1/2)=\log 2-\Hb(q),
\]
the relative-entropy cost of driving the Bernoulli parameter to the zero-adjustment point \(p=1/2\), where \(\gamma(p)=\log((1-p)/p)\) vanishes.
On the stable region, for \(q<1/2\) and \(0\le c<c_{\rm crit}(q)\), let \(p_c(q)\in[q,1/2)\) be the unique
solution of
\[
        \Hb(p_c(q))=\Hb(q)+c .
\]
Then, on all of \(\mathcal Q=[0,1]\),
\[
        f^*(c,q)
        =
        \begin{cases}
        p_c(q)-q, & q\in \mathcal Q_s,\ 0\le c<c_{\rm crit}(q),\\[0.4ex]
        1/2-q, & q\in \mathcal Q_s,\ c=c_{\rm crit}(q),\\[0.4ex]
        0, & q\notin \mathcal Q_s,\ c\le c_{\rm crit}(q),\\[0.4ex]
        +\infty, & c>c_{\rm crit}(q) .
        \end{cases}
\]
\end{proposition}

\noindent On the complement \(q\ge1/2\), no stable law explains the observation at rate below \(c_{\rm crit}(q)\), so the safety constraint \(I_p(q)+\gamma(p)f\ge c\) already holds at \(f=0\); the profile is therefore \(0\) until \(c\) reaches \(c_{\rm crit}(q)\), beyond which no finite buffer is safe.  The proof is given in Appendix~\ref{app:examples}.

\subsection{Gaussian increments with unknown mean and variance}\label{subsec:gaussian}

Let \(X\sim N(\mu,\sigma^2)\), where \(\mu<0\) and \(\sigma^2>0\), so that \(\E X=\mu<0\); the model class is \(\mathcal P=\{(\mu,\sigma^2):\mu<0,\sigma^2>0\}\), with adjustment coefficient \(\gamma(\mu,\sigma^2)=-2\mu/\sigma^2\).  We observe \(Q_n=(\bar X_n,V_n)\), where \(\bar X_n=n^{-1}\sum_{i=1}^n X_i\) and \(V_n=n^{-1}\sum_{i=1}^n (X_i-\bar X_n)^2\).  For Gaussian samples with at least two observations \(V_n>0\) almost surely, so the statistic space is \(\mathcal Q=\mathbb R\times(0,\infty)\), and we write \((m,v)\) for a generic value.

The Gaussian law fitted to \((m,v)\) is \(N(m,v)\); it is stable, that is, has negative drift, precisely when \(m<0\).  We call \(\mathcal Q_s=\{(m,v)\in\mathcal Q:m<0\}\) the \emph{stable region}; on its complement \(\mathcal Q\setminus\mathcal Q_s=\{(m,v)\in\mathcal Q:m\ge0\}\) the fitted Gaussian has nonnegative drift.

By Cram\'er's theorem applied to \((X_i,X_i^2)\) and the contraction principle \cite[Sec.~2.2 and Theorem~4.2.1]{dembo_zeitouni}, under \(N(\mu,\sigma^2)\) the statistic \(Q_n\) satisfies an LDP with rate
\[
        I_{\mu,\sigma^2}(m,v)
        =
        D\left(N(m,v)\|N(\mu,\sigma^2)\right),
\]
which can be written explicitly as (see e.g.\ \cite{kullback1959})
\[
        I_{\mu,\sigma^2}(m,v)
        =
        \frac{(m-\mu)^2}{2\sigma^2}
        +\frac12\left(
        \frac{v}{\sigma^2}-1-
        \log\frac{v}{\sigma^2}\right).
\]
Thus, for \((m,v)\in\mathcal Q\),
\[
        f^*(c,m,v)=
        \sup_{\mu<0,\sigma^2>0}
        \frac{(c-I_{\mu,\sigma^2}(m,v))_+}{-2\mu/\sigma^2}.
\]
\begin{theorem}[Explicit Gaussian profile]\label{thm:gaussian}
Let \(c>0\) and \((m,v)\in\mathcal Q\), put \(\Delta_c^{\rm g}=e^{2c}-1\), and let
\[
        c_{\rm crit}(m,v)=\tfrac12\log\!\left(1+\tfrac{m^2}{v}\right)
\]
be the critical exponent of Proposition~\ref{prop:phase}.  Then
\[
        f^*(c,m,v)=
        \begin{cases}
        \left(\dfrac{-m-\sqrt{m^2-v\Delta_c^{\rm g}}}{2}\right)_+,
        & c\le c_{\rm crit}(m,v),\\[12pt]
        +\infty,
        & c> c_{\rm crit}(m,v),
        \end{cases}
\]
where the square root is real precisely because \(c\le c_{\rm crit}(m,v)\) is equivalent to \(m^2\ge v\Delta_c^{\rm g}\).
\end{theorem}

The critical exponent \(c_{\rm crit}(m,v)\) is the relative-entropy cost of driving the fitted mean \(\mu\) to \(0\), where the adjustment coefficient \(-2\mu/\sigma^2\) vanishes.  On the stable region \(m<0\) the finite branch is the positive buffer \((-m-\sqrt{m^2-v\Delta_c^{\rm g}})/2\).  On the complement \(m\ge0\) the bracket is nonpositive, so \(f^*=0\) whenever \(c\le c_{\rm crit}(m,v)\): no stable law explains a nonnegative sample mean at rate below \(c_{\rm crit}(m,v)\), and the safety constraint already holds with zero buffer.  The proof is given in Appendix~\ref{app:examples}.

\subsection{Difference of two exponential increments}\label{subsec:expdiff}

Let \(Y\sim \Exp(\mu)\), \(Z\sim \Exp(\lambda)\), and \(X=Y-Z\), where \(Y\) and \(Z\) are independent.  The model class is \(\mathcal P=\{(\mu,\lambda):\mu>\lambda>0\}\), the stable case, in which \(\E X=1/\mu-1/\lambda<0\).  The moment generating function is
\[
        \E e^{\theta X}
        =
        \frac{\mu}{\mu-\theta}\frac{\lambda}{\lambda+\theta},
        \qquad -\lambda<\theta<\mu,
\]
and therefore the increment \(X=Y-Z\) has adjustment coefficient \(\gamma(\mu,\lambda)=\mu-\lambda\).  We assume that the decision maker observes the two components separately, so the observation model is the joint law \(P_{\mu,\lambda}\) of the pair \((Y,Z)\); the future ruin process \(M\) has increments distributed as \(X\), whose law is the image of \(P_{\mu,\lambda}\) under \((y,z)\mapsto y-z\).  Set \(Q_n=(\bar Y_n,\bar Z_n)\), where \(\bar Y_n=n^{-1}\sum_{i=1}^n Y_i\) and \(\bar Z_n=n^{-1}\sum_{i=1}^n Z_i\).  By Cram\'er's theorem applied to the independent vector \((Y_i,Z_i)\) \cite[Sec.~2.2]{dembo_zeitouni}, under \(P_{\mu,\lambda}\) the statistic \(Q_n\) satisfies an LDP with rate
\[
        I_{\mu,\lambda}(u,v)
        =
        \mu u-1-\log(\mu u)
        +
        \lambda v-1-\log(\lambda v),
        \qquad u,v>0.
\]
\begin{remark}
If only the differences \(X_i=Y_i-Z_i\) are observed, the likelihood and sufficient
statistics are different; we do not treat that case here.
\end{remark}

The means \(\bar Y_n,\bar Z_n\) are almost surely positive, so the statistic space is \(\mathcal Q=(0,\infty)^2\), with generic value \((u,v)\).  The law fitted to \((u,v)\) has rates \(\hat\mu=1/u\) and \(\hat\lambda=1/v\); it is stable, that is, has negative drift, precisely when \(u<v\).  We call \(\mathcal Q_s=\{(u,v)\in\mathcal Q:u<v\}\) the \emph{stable region}; on its complement \(\mathcal Q\setminus\mathcal Q_s=\{(u,v)\in\mathcal Q:u\ge v\}\) the fitted rates give nonnegative drift.  For \((u,v)\in\mathcal Q\),
\[
        f^*(c,u,v)
        =
        \sup_{\mu>\lambda>0}
        \frac{(c-I_{\mu,\lambda}(u,v))_+}{\mu-\lambda}.
\]

\begin{theorem}[Explicit profile for exponential differences]\label{thm:expdiff}
Let \(c>0\) and \((u,v)\in\mathcal Q\), put \(\Delta_c^{\rm e}=e^c-1\), and let
\[
        c_{\rm crit}(u,v)=\log\frac{(u+v)^2}{4uv}
\]
be the critical exponent of Proposition~\ref{prop:phase}.  Then
\[
        f^*(c,u,v)
        =
        \begin{cases}
        \left(\dfrac{v-u-\sqrt{(v-u)^2-4uv\Delta_c^{\rm e}}}{2}\right)_+,
        & c\le c_{\rm crit}(u,v),\\[12pt]
        +\infty,
        & c> c_{\rm crit}(u,v),
        \end{cases}
\]
where the square root is real precisely because \(c\le c_{\rm crit}(u,v)\) is equivalent to \((v-u)^2\ge4uv\Delta_c^{\rm e}\).
\end{theorem}

The critical exponent \(c_{\rm crit}(u,v)\) is the relative-entropy cost of driving the fitted rates together, \(\mu\downarrow\lambda\), where the adjustment coefficient \(\mu-\lambda\) vanishes.  On the stable region \(u<v\) the finite branch is the positive buffer \((v-u-\sqrt{(v-u)^2-4uv\Delta_c^{\rm e}})/2\).  On the complement \(u\ge v\) the bracket is nonpositive, so \(f^*=0\) whenever \(c\le c_{\rm crit}(u,v)\): no stable law explains an observation with \(\bar Y_n\ge\bar Z_n\) at rate below \(c_{\rm crit}(u,v)\), and the safety constraint already holds with zero buffer.  The proof is given in Appendix~\ref{app:examples}.

\section{Nonparametric envelope classes}\label{sec:nonparametric}

We now turn to nonparametric model classes. The central difficulty is
uniform control of the upper tail. For the worst-case adjustment
coefficient to remain genuinely data-dependent, the model class must
exclude laws that place a vanishing amount of probability at
increasingly distant right-tail locations. 
A single exponential constraint \(\E_P e^{aX}\le1\) does not enforce it, and Section~\ref{sec:unit-expon-envel} shows that the resulting profile then degenerates to a data-independent rule.  One additional, superlinear exponential moment restores uniform integrability, and Section~\ref{subsec:two-level-envelope} builds from it a well-posed, genuinely data-dependent safe rule.

In this section the decision maker observes the increments \(X_1,\dots,X_n\) directly, and the statistic is their empirical distribution \(Q_n=n^{-1}\sum_{i=1}^n\delta_{X_i}\) (the empirical law \(P_n\) of Section~\ref{sec:problem-formulation}).  The statistic space \(\mathcal Q\) is then a space of probability laws, and we write \(Q\) for a generic value, interpreted as a candidate empirical law.  For this choice of statistic, Sanov's theorem motivates identifying the large-deviation rate of the empirical distribution under a law \(P\) with the relative entropy \(D(Q\|P)\), the information cost that enters the decision-centric profile.  For a class \(\mathcal C\) of increment laws, this profile is that of Definition~\ref{def:fstar} with the empirical law as statistic,
\begin{equation}
\label{eq-ansatz}
        f^*_{\mathcal C}(c,Q)
        =
        \sup_{P\in\mathcal C}
        \frac{(c-D(Q\|P))_+}{\gamma(P)} .
\end{equation}
This section compares two unrestricted nonparametric choices of \(\mathcal C\).
We note that the form of $f^*_{\mathcal C}$ in this case should only be seen as a candidate solution: in the level of generality considered in this section, we are not aware of a suitable version of Sanov's theorem. Thus, while Assumption~\ref{ass:LDP} and Sanov's theorem for empirical measures can be seen as an ansatz leading to the candidate profile (\ref{eq-ansatz}), we analyze its properties directly in this section.

\subsection{The unit exponential envelope is degenerate}
\label{sec:unit-expon-envel}

The simplest nonparametric envelope imposes a single exponential moment bound.  Fix any \(r>0\) and consider the unit envelope class \(\mathcal P_r=\{P:\E_Pe^{rX}\le1\}\); every law in it has \(\gamma(P)\ge r\), so this one constraint already guarantees the adjustment coefficient stays above the floor \(r\).  Define the projection cost \(H_r(Q):=\inf_{P\in\mathcal P_r}D(Q\|P)\), with the convention that the infimum over the empty set is \(+\infty\).
The projection enjoys a convenient dual characterization, reused at exponent \(b\) in the next subsection.  For \(\xi=(\lambda,\beta)\in\mathbb R_+^2\) write \(s^r_\xi(x):=\lambda+\beta e^{rx}\), let \(\mathcal N:=\{(\lambda,\beta)\in\mathbb R_+^2:\lambda+\beta=1\}\) be the (compact) unit simplex, and define
\[
        H^d_r(Q):=\sup_{\xi\in\mathcal N}\int\log s^r_\xi\,dQ,
\]
with the extended-real convention that the logarithmic integral is \(-\infty\) if \(s^r_\xi\) vanishes on a set of positive \(Q\)-mass.

\begin{lemma}[Duality for the unit-envelope projection]\label{lem:unit-dual}
For every \(r>0\) and every probability law \(Q\) on \(\mathbb R\), one has \(H^d_r(Q)\le H_r(Q)\).  If moreover \(\E_Q e^{rX}<\infty\), then \(H^d_r(Q)=H_r(Q)\).
\end{lemma}

\begin{proof}
\emph{Weak duality.}  Let \(P\in\mathcal P_r\) and \(\xi=(\lambda,\beta)\in\mathcal N\).  Since \(\lambda+\beta=1\) and \(\E_P e^{rX}\le1\),
\[
        \E_P s^r_\xi(X)=\lambda+\beta\,\E_P e^{rX}\le\lambda+\beta=1.
\]
The entropy variational inequality \(D(Q\|P)\ge\int\log h\,dQ-\log\int h\,dP\), valid for every nonnegative measurable \(h\), applied to \(h=s^r_\xi\) gives
\[
        D(Q\|P)\ge\int\log s^r_\xi\,dQ-\log\E_P s^r_\xi(X)\ge\int\log s^r_\xi\,dQ,
\]
since \(\log\E_P s^r_\xi\le\log1=0\).  Taking the infimum over \(P\in\mathcal P_r\) and the supremum over \(\xi\in\mathcal N\) proves \(H^d_r(Q)\le H_r(Q)\).

\emph{Relaxation.}  We first relax the primal to nonnegative measures under an inequality normalization,
\[
        \widetilde H_r(Q):=\inf\Big\{D(Q\|P):P\ge0,\ Q\ll P,\ \textstyle\int dP\le1,\ \int e^{rX}\,dP\le1\Big\},
\]
and show \(\widetilde H_r(Q)=H_r(Q)\).  Clearly \(\widetilde H_r\le H_r\).  For the reverse, let \(P\) be feasible for \(\widetilde H_r\) and \(\varepsilon\in(0,1)\); scaling by \(1-\varepsilon\) opens strict slack in the moment constraint, so the probability law
\[
        P_\varepsilon:=(1-\varepsilon)P+\kappa\,\delta_{-L},\qquad
        \kappa:=1-(1-\varepsilon)\textstyle\int dP\in[\varepsilon,1],
\]
with \(L\) large enough that \(\kappa e^{-rL}\le\varepsilon\), satisfies \(\int e^{rX}\,dP_\varepsilon=(1-\varepsilon)\int e^{rX}\,dP+\kappa e^{-rL}\le(1-\varepsilon)+\varepsilon=1\), so \(P_\varepsilon\in\mathcal P_r\).  Since \(P_\varepsilon\ge(1-\varepsilon)P\) and \(D(Q\|\cdot)\) is nonincreasing in its second argument, \(D(Q\|P_\varepsilon)\le D(Q\|(1-\varepsilon)P)=D(Q\|P)-\log(1-\varepsilon)\); the infimum over \(P\) followed by \(\varepsilon\downarrow0\) gives \(H_r\le\widetilde H_r\).

\emph{Strong duality.}  Assume \(\E_Q e^{rX}<\infty\).  The relaxed program minimizes the convex functional \(P\mapsto D(Q\|P)\) subject to two affine inequality constraints valued in \(\mathbb R^2\), and \(P_S=tQ\) with \(t\in(0,1)\) small enough that \(t\max\{1,\E_Q e^{rX}\}<1\) is strictly feasible with \(D(Q\|P_S)=-\log t<\infty\).  By the Lagrange duality theorem for convex programs with finitely many constraints admitting a Slater point \cite[Section~8.6]{luenberger1969}, the dual value is attained and equals the primal,
\[
        H_r(Q)=\max_{\lambda,\beta\ge0}\Big\{\inf_{P\ge0}\big[D(Q\|P)+\textstyle\int(\lambda+\beta e^{rX})\,dP\big]-\lambda-\beta\Big\}.
\]
To evaluate the inner infimum over measures, parametrize \(P\) by its Lebesgue decomposition relative to \(Q\): by the Radon--Nikodym and Lebesgue decomposition theorems, \(\dd P=w\,\dd Q+\dd P^\perp\) with \(w:=\dd P/\dd Q\ge0\) measurable and \(P^\perp\ge0\) singular to \(Q\), and \((w,P^\perp)\) ranges over all such pairs.  The objective then splits as
\[
        D(Q\|P)+\int s^r_\xi\,dP=\int\big(s^r_\xi\,w-\log w\big)\,dQ+\int s^r_\xi\,dP^\perp .
\]
Since \(s^r_\xi\ge0\), the singular term is minimized by \(P^\perp=0\).  For the absolutely continuous part, the elementary inequality \(sv-\log v\ge1+\log s\) for all \(v>0\), with equality at \(v=1/s\), integrates to \(\int(s^r_\xi w-\log w)\,dQ\ge\int(1+\log s^r_\xi)\,dQ\) for every admissible \(w\), and the measurable density \(w=1/s^r_\xi\) attains it.  Hence the infimum equals \(\int(1+\log s^r_\xi)\,dQ=\int\log s^r_\xi\,dQ+1\), using \(\int\dd Q=1\), attained at \(\dd P=\dd Q/s^r_\xi\).  The dual is therefore \(\max_{\lambda,\beta\ge0}[\int\log(\lambda+\beta e^{rX})\,dQ+1-\lambda-\beta]\); optimizing along each ray \((\lambda,\beta)\mapsto(t\lambda,t\beta)\) removes the affine term and normalizes \(\lambda+\beta=1\), leaving \(\sup_{\xi\in\mathcal N}\int\log s^r_\xi\,dQ=H^d_r(Q)\).  Hence \(H^d_r(Q)=H_r(Q)\).
\end{proof}

The floor \(r\) is only a worst-case guarantee on the adjustment coefficient.  If the data are generated by a law \(P\) with \(\gamma(P)>r\) strictly, one might hope that one may learn a smaller safe capital buffer.  The following proposition shows that this hope is unfounded: under the unrestricted unit envelope the worst-case profile always behaves as if \(\gamma=r\), and the data enter only through the projection cost \(H_r(Q)\).

\begin{proposition}[Boundary collapse under the unit envelope]\label{prop:unit-envelope-collapse}
\[
        f^*_{\mathcal P_r}(c,Q)
        =
        \frac{(c-H_r(Q))_+}{r}.
\]
\end{proposition}

\begin{proof}
The upper bound is immediate.  Since \(\gamma(P)\ge r\) for every
\(P\in\mathcal P_r\),
\[
        \frac{(c-D(Q\|P))_+}{\gamma(P)}
        \le
        \frac{(c-D(Q\|P))_+}{r}
        \le
        \frac{(c-H_r(Q))_+}{r}.
\]

It remains to prove the reverse inequality.  Let \(\varepsilon>0\) and choose
\(P\in\mathcal P_r\) such that \(D(Q\|P)\le H_r(Q)+\varepsilon\).  Fix \(\kappa\in(0,1)\).  If \(\E_Pe^{rX}=1\), choose \(L>0\) and put \(\widetilde P=(1-\kappa)P+\kappa\delta_{-L}\).  Then
\[
        \E_{\widetilde P}e^{rX}
        =(1-\kappa)\E_Pe^{rX}+\kappa e^{-rL}<1,
\]
and \(\widetilde P\ge(1-\kappa)P\).  Since \(D(Q\|P)<\infty\), we have
\(Q\ll P\), and hence \(Q\ll\widetilde P\).  Therefore
\[
        D(Q\|\widetilde P)
        \le
        D(Q\|P)-\log(1-\kappa).
\]
Choosing \(\kappa\) so small that \(-\log(1-\kappa)<\varepsilon\), and replacing
\(P\) by \(\widetilde P\), we may assume, at the cost of increasing
\(D(Q\|P)\) by at most \(\varepsilon\), that \(m_r:=\E_Pe^{rX}<1\).

Let \(\eta>0\).  For \(z>0\), put \(s_z=((1-m_r)/2)e^{-rz}\) and \(P_z=(1-s_z)P+s_z\delta_z\).  Then
\[
        \E_{P_z}e^{rX}
        =(1-s_z)m_r+s_ze^{rz}
        \le \frac{1+m_r}{2}<1,
\]
so \(P_z\in\mathcal P_r\).  Moreover \(P_z\ge(1-s_z)P\), and hence
\[
  D(Q\|P_z)
  \le
  D(Q\|P)-\log(1-s_z)
  \longrightarrow D(Q\|P)
\]
for $z\to\infty$.
On the other hand,
\[
        \E_{P_z}e^{(r+\eta)X}
        \ge s_ze^{(r+\eta)z} \to\infty .
\]
Thus, for all large \(z\), the adjustment coefficient of \(P_z\) is at most
\(r+\eta\).  Therefore
\[
        f^*_{\mathcal P_r}(c,Q)
        \ge
        \frac{(c-D(Q\|P_z))_+}{r+\eta}.
\]
Letting \(z\to\infty\), then \(\varepsilon\downarrow0\), and finally
\(\eta\downarrow0\), gives the lower bound.
\end{proof}

The proposition shows that the data enter the unrestricted unit-envelope problem
only through the projection cost \(H_r(Q)\); the adjustment coefficient itself
is always driven to the floor \(r\) by arbitrarily small unseen right-tail mass.  This is exactly what defeats the hope raised above.  Suppose the data are generated by a law \(P\) with \(\gamma(P)>r\), equivalently \(\E_Pe^{rX}<1\), so that \(P\) lies strictly inside the envelope.  By the law of large numbers \(n^{-1}\sum_{i=1}^n e^{rX_i}\to\E_Pe^{rX}<1\), so the empirical law satisfies \(Q_n\in\mathcal P_r\) eventually almost surely and hence \(H_r(Q_n)=0\).  Hence, with probability one, for all large \(n\) the projection cost vanishes and the rule collapses to the data-independent coefficient
\[
  f^*_{\mathcal P_r}(c,Q_n)=\frac{c}{r}.
\]

The remedy is to add uniform control of the upper tail, the subject of the next subsection.  In a sense this is unsurprising. Already for a single law, the standing assumptions of Section~\ref{sec:problem-formulation} required tail control of exactly this kind, i.e., finiteness of \(\Lambda_P\) on \([0,b]\) and hence \(\E_P e^{bX}<\infty\). The two-level envelope of the next subsection simply turns this per-law finiteness into a bound \(\E_P e^{bX}\le B\) that is uniform across the class.

\subsection{A two-level envelope and a conservative dual rule}
\label{subsec:two-level-envelope}

Proposition~\ref{prop:unit-envelope-collapse} identifies the precise obstruction created by the
single exponential envelope: a sequence of laws may place
vanishingly small mass at locations tending to \(+\infty\), while
preserving the constraint
\(\mathbb E_P e^{aX}\leq1\). Although this perturbation has negligible
relative-entropy cost, it can drive the adjustment coefficient down to
the boundary value \(a\).
To exclude this mechanism, the family
$
\bigl\{e^{aX}:P\in\mathcal P\bigr\}
$
must have uniformly negligible tails. A direct way to obtain this is
to impose one stronger exponential-moment bound. Indeed, if \(b>a\)
and
\[
\sup_{P\in\mathcal P}\mathbb E_P e^{bX}\leq B,
\]
then, for every \(K\geq1\),
\[
\sup_{P\in\mathcal P}
\mathbb E_P\left[
    e^{aX}\mathbf 1_{\{e^{aX}>K\}}
\right]
\leq
K^{1-b/a}
\sup_{P\in\mathcal P}\mathbb E_P e^{bX}  \leq
B K^{1-b/a}
\longrightarrow 0
\qquad\text{as }K\to\infty.
\]
Thus the higher exponential moment makes
\(\{e^{aX}:P\in\mathcal P\}\) uniformly integrable and rules out the
escaping-spike construction. This is the power-function instance of
the de la Vall\'ee--Poussin criterion for uniform integrability; see,
for example, \cite[Corollary~6.21]{Klenke2014}. Accordingly, fix
\(0<a<b\) and \(B>1\), and define
\[
\mathcal P_{a,b,B}
:=
\left\{
P:
\mathbb E_Pe^{aX}\leq1,\quad
\mathbb E_Pe^{bX}\leq B
\right\}.
\]
The second constraint gives the uniform  estimate \(P(X>x)\le B e^{-bx}\), \(x\ge0\), and thus rules out the spike mechanism that made the unit-envelope formulation degenerate in Section \ref{sec:unit-expon-envel}.

For \(a<y<b\), define
\[
        H_{a,b,B}(Q,y)
        :=
        \inf\left\{
        D(Q\|P):
        P\in\mathcal P_{a,b,B},\,
        \E_Pe^{yX}\ge1
        \right\},
\]
with the convention that the infimum over the empty set is \(+\infty\); this is the entropy cost of reaching adjustment coefficients not larger than \(y\).
On \(\{\gamma(P)\ge b\}\) the denominator in the profile is bounded below by \(b\), and this set is exactly the unit envelope of Section~\ref{sec:unit-expon-envel} at exponent \(b\): the condition \(\gamma(P)\ge b\) is equivalent to \(\E_P e^{bX}\le1\), and since \(a<b\) and \(B>1\) this makes both constraints defining \(\mathcal P_{a,b,B}\) automatic, so
\[
        \{P\in\mathcal P_{a,b,B}:\gamma(P)\ge b\}=\{P:\E_P e^{bX}\le1\}=\mathcal P_b .
\]
The endpoint projection is therefore \(H_b(Q)\), the unit-envelope projection at exponent \(b\).

\begin{proposition}[Profile under the two-level envelope]
\label{prop:two-level-envelope-profile}
For every \(c>0\) and every probability law \(Q\),
\[
        f^*_{\mathcal P_{a,b,B}}(c,Q)
        =
        \max\left\{
        \sup_{a<y<b}
        \frac{(c-H_{a,b,B}(Q,y))_+}{y},
        \frac{(c-H_b(Q))_+}{b}
        \right\}.
\]
\end{proposition}

\begin{proof}
Split the defining supremum
\[
        f^*_{\mathcal P_{a,b,B}}(c,Q)
        =
        \sup_{P\in\mathcal P_{a,b,B}}
        \frac{(c-D(Q\|P))_+}{\gamma(P)}
\]
into the cases \(\gamma(P)<b\) and \(\gamma(P)\ge b\).

First suppose \(\gamma(P)<b\); we establish the interior identity
\begin{equation}\label{eq:interior-identity}
        \sup_{\substack{P\in\mathcal P_{a,b,B}\\ \gamma(P)<b}}
        \frac{(c-D(Q\|P))_+}{\gamma(P)}
        =
        \sup_{a<y<b}
        \frac{(c-H_{a,b,B}(Q,y))_+}{y}.
\end{equation}
To prove ``\(\le\)'', fix \(P\) with \(\gamma(P)<b\).  For every \(y\in(\gamma(P),b)\), one has
\(\E_Pe^{yX}>1\), so \(P\) is feasible in the definition of \(H_{a,b,B}(Q,y)\), and hence \(H_{a,b,B}(Q,y)\le D(Q\|P)\); equivalently \((c-H_{a,b,B}(Q,y))_+\ge (c-D(Q\|P))_+\).  This numerator bound holds for every \(y\in(\gamma(P),b)\), whereas the denominator \(y\) exceeds \(\gamma(P)\) and only recovers it in the limit \(y\downarrow\gamma(P)\).  Writing the target as such a limit,
\[
        \frac{(c-D(Q\|P))_+}{\gamma(P)}
        =
        \liminf_{y\downarrow\gamma(P)}
        \frac{(c-D(Q\|P))_+}{y}
        \le
        \liminf_{y\downarrow\gamma(P)}
        \frac{(c-H_{a,b,B}(Q,y))_+}{y}
        \le
        \sup_{a<y<b}
        \frac{(c-H_{a,b,B}(Q,y))_+}{y},
\]
where the first equality holds because the numerator does not depend on \(y\) and \(y\to\gamma(P)\), and the first inequality is the termwise numerator bound.  Taking the supremum over \(P\) with \(\gamma(P)<b\) yields ``\(\le\)'' in~\eqref{eq:interior-identity}.

For the reverse inequality, fix \(a<y<b\).  Every \(P\in\mathcal P_{a,b,B}\) with \(\E_Pe^{yX}\ge1\) has \(\gamma(P)\le y<b\), so
\[
        \frac{(c-D(Q\|P))_+}{\gamma(P)}\ge\frac{(c-D(Q\|P))_+}{y}.
\]
Take the supremum over these \(P\).  The left-hand side is then at most \(\sup_{P:\gamma(P)<b}(c-D(Q\|P))_+/\gamma(P)\); the right-hand side equals \((c-H_{a,b,B}(Q,y))_+/y\), because \(H_{a,b,B}(Q,y)=\inf\{D(Q\|P):P\in\mathcal P_{a,b,B},\,\E_Pe^{yX}\ge1\}\).  Taking the supremum over \(a<y<b\) yields ``\(\ge\)'' in~\eqref{eq:interior-identity}, completing the proof of the interior identity.

It remains to identify the endpoint contribution.  As recorded before the statement, the two constraints defining \(\mathcal P_{a,b,B}\) become automatic once \(\gamma(P)\ge b\), so that \(\{P\in\mathcal P_{a,b,B}:\gamma(P)\ge b\}=\mathcal P_b\), the unit envelope at exponent \(b\).  On this set the endpoint contribution is exactly the unit-envelope profile at exponent \(b\), which Proposition~\ref{prop:unit-envelope-collapse} evaluates in closed form:
\[
        \sup_{\substack{P\in\mathcal P_{a,b,B}\\ \gamma(P)\ge b}}
        \frac{(c-D(Q\|P))_+}{\gamma(P)}
        =
        f^*_{\mathcal P_b}(c,Q)
        =
        \frac{(c-H_b(Q))_+}{b}.
\]
Combining this with the interior identity~\eqref{eq:interior-identity} gives the claimed maximum, completing the proof.
\end{proof}

We next derive conservative lower bounds on these entropy projections.  The
bounds are dual in nature: for each fixed \(y\in(a,b)\), they involve finitely
many dual variables and a one-dimensional positivity constraint.  This
constraint is still explicit, but it is much simpler than the original
optimization over probability laws.
To make this concrete, we set, for \(a<y<b\) and
\(\zeta=(\lambda,\alpha,\beta,\eta)\in\mathbb R_+^4\),
\[
        s_{\zeta,y}(x)
        :=
        \lambda+
        \alpha e^{ax}+
        \beta e^{bx}-
        \eta e^{yx},
\]
and let
\[
        \mathcal N_y:=
        \{\zeta\in\mathbb R_+^4:
        s_{\zeta,y}(x)\ge0\text{ for all }x\in\mathbb R,\ \
        \lambda+\alpha+\beta B-\eta=1\}.
\]
Define the interior dual directly by its compact representation,
\[
        H^d_{a,b,B}(Q,y)
        :=
        \sup_{\zeta\in\mathcal N_y}
        \int\log s_{\zeta,y}\,dQ.
\]
For the endpoint we reuse the unit-envelope dual of Section~\ref{sec:unit-expon-envel} at exponent \(b\), writing \(H^d_b(Q)\) for that dual and \(s^b_\xi\) for its affine integrand; its logarithmic integral, like that of \(H^d_{a,b,B}\), is read in the extended-real sense introduced there.

\begin{lemma}[Compactness and duality for the interior projection]\label{lem:compact-interior}
For every \(a<y<b\) and every probability law \(Q\) on \(\mathbb R\), the set \(\mathcal N_y\) is compact and convex, and \(H^d_{a,b,B}(Q,y)\le H_{a,b,B}(Q,y)\).  If moreover \(\E_Q e^{bX}<\infty\), then \(H^d_{a,b,B}(Q,y)=H_{a,b,B}(Q,y)\).
\end{lemma}

\begin{proof}
\emph{Weak duality.}  Let \(\zeta\in\mathcal N_y\) and let \(P\) be feasible in the definition of \(H_{a,b,B}(Q,y)\), so \(\E_P e^{aX}\le1\), \(\E_P e^{bX}\le B\), and \(\E_P e^{yX}\ge1\).  Since \(\alpha,\beta,\eta\ge0\),
\[
        \E_P s_{\zeta,y}(X)
        =\lambda+\alpha\E_P e^{aX}+\beta\E_P e^{bX}-\eta\E_P e^{yX}
        \le\lambda+\alpha+\beta B-\eta=1.
\]
The entropy variational inequality \(D(Q\|P)\ge\int\log h\,dQ-\log\int h\,dP\), applied to \(h=s_{\zeta,y}\), gives \(D(Q\|P)\ge\int\log s_{\zeta,y}\,dQ-\log\E_P s_{\zeta,y}\ge\int\log s_{\zeta,y}\,dQ\), since \(\log\E_P s_{\zeta,y}\le\log1=0\).  Taking the infimum over \(P\) and the supremum over \(\zeta\in\mathcal N_y\) proves \(H^d_{a,b,B}(Q,y)\le H_{a,b,B}(Q,y)\).

\emph{Compactness.}  For \(\zeta\in\mathcal N_y\) we have \(\lambda,\alpha,\beta\ge0\) by definition.  Choose \(x_0>0\) with \(e^{(b-y)x_0}<B\).  Writing \(\eta=\lambda+\alpha+\beta B-1\),
\[
        s_{\zeta,y}(x_0)
        =e^{yx_0}-\lambda(e^{yx_0}-1)-\alpha(e^{yx_0}-e^{ax_0})-\beta(Be^{yx_0}-e^{bx_0}),
\]
where the three coefficients after the minus signs are strictly positive.  With \(\lambda,\alpha,\beta\ge0\) and \(s_{\zeta,y}(x_0)\ge0\), this bounds \(\lambda,\alpha,\beta\), and the defining equation \(\eta=\lambda+\alpha+\beta B-1\) then bounds \(\eta\); thus \(\mathcal N_y\) is bounded.  It is closed and convex, being defined by the linear equation \(\lambda+\alpha+\beta B-\eta=1\) together with the inequalities \(s_{\zeta,y}(x)\ge0\), each of which is linear in \(\zeta\).  Hence \(\mathcal N_y\) is compact and convex.

\emph{Relaxation.}  As in Lemma~\ref{lem:unit-dual}, relax the primal to nonnegative measures under an inequality normalization,
\[
        \widetilde H_{a,b,B}(Q,y):=\inf\Big\{D(Q\|P):P\ge0,\ Q\ll P,\ \textstyle\int dP\le1,\ \int e^{aX}dP\le1,\ \int e^{bX}dP\le B,\ \int e^{yX}dP\ge1\Big\},
\]
and show \(\widetilde H_{a,b,B}(Q,y)=H_{a,b,B}(Q,y)\).  Clearly ``\(\le\)'' holds.  For the reverse, scaling alone would destroy the slack in the lower constraint \(\int e^{yX}dP\ge1\), so we mix with a strictly feasible probability law.  Fix \(x_0\in(0,\tfrac{\log B}{b-y})\) and \(c\in(e^{-yx_0},e^{-ax_0})\) with \(c\,e^{bx_0}<B\); for \(L'\) large the probability law \(R:=c\,\delta_{x_0}+(1-c)\,\delta_{-L'}\) then satisfies \(\int e^{aX}dR<1\), \(\int e^{bX}dR<B\) and \(\int e^{yX}dR>1\), all strictly.  Given \(P\) feasible for \(\widetilde H_{a,b,B}\) and \(\theta\in(0,1)\), define
\[
        P_\theta:=(1-\theta)P+\theta R+(1-\theta)\big(1-\textstyle\int dP\big)\delta_{-L};
\]
its total mass is \((1-\theta)\int dP+\theta+(1-\theta)(1-\int dP)=1\), the third term carrying exactly the normalization deficit of the subprobability \(P\), so \(P_\theta\) is a probability law.  We check that its three moment constraints hold for every \(\theta\in(0,1)\), on taking \(L=L(P,\theta)\) large.  The lower bound survives the contraction by \(1-\theta\): indeed \(\int e^{yX}dP_\theta\ge(1-\theta)\int e^{yX}dP+\theta\int e^{yX}dR>1\), for any \(L\).  The two upper bounds are instead the ones the far-left mass at \(-L\) pushes on, contributing \((1-\theta)(1-\int dP)e^{-aL}\) and \((1-\theta)(1-\int dP)e^{-bL}\) to the respective moments; since \(R\) is strictly feasible in both, choosing \(L\) large enough that these contributions fall within its slack keeps \(\int e^{aX}dP_\theta\le1\) and \(\int e^{bX}dP_\theta\le B\).  Hence \(P_\theta\in\mathcal P_{a,b,B}\) with \(\int e^{yX}dP_\theta\ge1\).  Since \(P_\theta\ge(1-\theta)P\), \(D(Q\|P_\theta)\le D(Q\|P)-\log(1-\theta)\), a bound independent of \(L\); the infimum over \(P\) followed by \(\theta\downarrow0\) gives \(H_{a,b,B}\le\widetilde H_{a,b,B}\).

\emph{Strong duality.}  Assume \(\E_Q e^{bX}<\infty\).  The relaxed program minimizes the convex functional \(P\mapsto D(Q\|P)\) subject to four affine inequality constraints valued in \(\mathbb R^4\), and \(P_S=c\,\delta_{x_0}+tQ\) (with \(x_0,c\) as above) is a Slater point: \(c\,\delta_{x_0}\) satisfies the four constraints strictly, including \(\int dP_S<1\) since \(c<1\).  Adding \(tQ\) secures \(Q\ll P_S\) with \(D(Q\|P_S)\le-\log t<\infty\); and since \(\E_Q e^{bX}<\infty\) (and hence \(\E_Q e^{aX},\E_Q e^{yX}<\infty\) as \(a,y<b\)) the four constraint values shift by the finite amounts \(t\), \(t\,\E_Q e^{aX}\), \(t\,\E_Q e^{bX}\), \(t\,\E_Q e^{yX}\), which vanish as \(t\downarrow0\), so a small \(t\) keeps all four inequalities strict.  By the Lagrange duality theorem \cite[Section~8.6]{luenberger1969}, the dual value is attained and equals the primal,
\[
        H_{a,b,B}(Q,y)=\max_{\lambda,\alpha,\beta,\eta\ge0}\Big\{\inf_{P\ge0}\big[D(Q\|P)+\textstyle\int s_{\zeta,y}\,dP\big]-(\lambda+\alpha+\beta B-\eta)\Big\}.
\]
As in Lemma~\ref{lem:unit-dual}, decomposing \(P\) relative to \(Q\) reduces the inner infimum to \(\int(s_{\zeta,y}\,w-\log w)\,dQ+\int s_{\zeta,y}\,dP^\perp\).  The dual variable \(\zeta\) ranges over the orthant \(\mathbb R_+^4\), on which \(s_{\zeta,y}\) may become negative; wherever \(s_{\zeta,y}(x)<0\), placing singular mass at \(x\) sends \(\int s_{\zeta,y}\,dP^\perp\to-\infty\), so the inner infimum, and hence the dual objective, is then \(-\infty\).  This enforces the pointwise positivity \(s_{\zeta,y}\ge0\): only \(\zeta\) satisfying it can be dual-optimal, so the dual variables may be selected from \(\mathcal N_y\), and for them the infimum equals \(\int\log s_{\zeta,y}\,dQ+1\) as before.  The dual is therefore \(\max_{\zeta\ge0,\ s_{\zeta,y}\ge0}\big[\int\log s_{\zeta,y}\,dQ+1-N\big]\) with \(N:=\lambda+\alpha+\beta B-\eta\).  This maximum is finite (it equals \(H_{a,b,B}(Q,y)\)), so any maximizer has \(N>0\): on a ray with \(N\le0\) and \(\int\log s_{\zeta,y}\,dQ>-\infty\), the scaling \(\zeta\mapsto t\zeta\) with \(t\to\infty\) would send the objective to \(+\infty\).  Optimizing the scale on a ray with \(N>0\) removes the affine term and normalizes \(N=1\), leaving \(\sup_{\zeta\in\mathcal N_y}\int\log s_{\zeta,y}\,dQ=H^d_{a,b,B}(Q,y)\).  Hence \(H^d_{a,b,B}(Q,y)=H_{a,b,B}(Q,y)\).
\end{proof}

Replacing each projection in Proposition~\ref{prop:two-level-envelope-profile} by its dual lower bound yields the \emph{conservative dual profile}
\[
        f_{\rm cd}(c,Q)
        :=
        \max\left\{
        \sup_{a<y<b}\frac{(c-H^d_{a,b,B}(Q,y))_+}{y},
        \frac{(c-H^d_b(Q))_+}{b}
        \right\}.
\]
Empirical laws \(Q_n\) have finite support, hence \(\E_{Q_n}e^{bX}<\infty\), so the equality parts of Lemmas~\ref{lem:compact-interior} and~\ref{lem:unit-dual}, with Proposition~\ref{prop:two-level-envelope-profile}, yield \(f_{\rm cd}(c,Q_n)=f^*_{\mathcal P_{a,b,B}}(c,Q_n)\): there the conservative dual rule coincides with the exact decision-centric profile.  With the pointwise lower bound of Section~\ref{sec:safe-profiles}, it is then optimal among regular rules.
By the weak-duality bounds \(H^d_{a,b,B}(Q,y)\le H_{a,b,B}(Q,y)\) (Lemma~\ref{lem:compact-interior}) and \(H^d_b(Q)\le H_b(Q)\) (Lemma~\ref{lem:unit-dual} at exponent \(b\)), Proposition~\ref{prop:two-level-envelope-profile} gives \(f_{\rm cd}(c,Q)\ge f^*_{\mathcal P_{a,b,B}}(c,Q)\) for every \(c>0\) and every probability law \(Q\); thus \(f_{\rm cd}\) dominates \(f^*\) and is a natural candidate.  We now show it is indeed safe, circumventing the technical difficulties of Sanov's theorem with a direct proof.

\begin{proposition}[Safety of the conservative dual rule]
\label{prop:full-dual-relaxed-safe}
For every \(c>0\) and every \(P\in\mathcal P_{a,b,B}\),
\[
        \limsup_{n\to\infty}
        \frac1n\log \Prob_P\{M>n f_{\rm cd}(c,Q_n)\}
        \le -c .
\]
Consequently \(C_n=n f_{\rm cd}(c,Q_n)\) is safe at exponent \(c\) over
\(\mathcal P_{a,b,B}\).
\end{proposition}

\begin{proof}
Evaluated at the empirical law \(Q_n=n^{-1}\sum_{t=1}^n\delta_{X_t}\), the two dual values are the empirical suprema
\begin{equation}\label{eq:empirical-duals}
        H^d_{a,b,B}(Q_n,y)=\sup_{\zeta\in\mathcal N_y}\frac1n\sum_{t=1}^n\log s_{\zeta,y}(X_t),
        \qquad
        H^d_b(Q_n)=\sup_{\xi\in\mathcal N}\frac1n\sum_{t=1}^n\log s^b_\xi(X_t),
\end{equation}
over the compact convex sets \(\mathcal N_y\) (Lemma~\ref{lem:compact-interior}) and \(\mathcal N\).  Since the historical sample is independent of the future maximum \(M\), Lundberg's inequality \eqref{eq:lundberg-bound} conditionally on \(Q_n\), followed by taking expectations, gives with \(\gamma=\gamma(P)\)
\[
        \Prob_P\{M>n f_{\rm cd}(c,Q_n)\}
        \le
        \E_P\,e^{-n\gamma f_{\rm cd}(c,Q_n)} .
\]

The integrand \(s_{\zeta,y}(x)=\lambda+\alpha e^{ax}+\beta e^{bx}-\eta e^{yx}\) is \emph{affine} in \(\zeta\), so for every \(x\) and every \(\theta\in(0,1]\) the map \(\zeta\mapsto s_{\zeta,y}(x)^\theta\) is concave; equivalently \(\zeta\mapsto\theta\log s_{\zeta,y}(x)\) is exp-concave.  We may therefore control the empirical suprema by the aggregation inequality \cite[Lemma~E.1]{agrawal-juneja-glynn-heavy}: if \(g_1,\dots,g_n\) are exp-concave on a compact convex \(\Lambda\subseteq\mathbb R^d\) and \(\delta_\Lambda\) is the uniform law on \(\Lambda\), i.e., the normalized Hausdorff measure on its affine hull, then, on every sample path,
\[
        \max_{\lambda\in\Lambda}\sum_{t=1}^n g_t(\lambda)
        \le
        \log\E_{\lambda\sim\delta_\Lambda}e^{\sum_{t=1}^n g_t(\lambda)}+d\log(n+1)+1 .
\]

\emph{Case \(\gamma<b\).}  Fix \(y\in(\gamma,b)\) and put \(\theta=\gamma/y\in(0,1)\).  Since \(f_{\rm cd}(c,Q)\ge(c-H^d_{a,b,B}(Q,y))_+/y\),
\[
        e^{-n\gamma f_{\rm cd}(c,Q_n)}
        \le
        e^{-n\theta(c-H^d_{a,b,B}(Q_n,y))_+}
        \le
        e^{-n\theta c}\,e^{\,n\theta H^d_{a,b,B}(Q_n,y)} .
\]
We now bound \(\E_P\,e^{\,n\theta H^d_{a,b,B}(Q_n,y)}\). By the empirical representation \eqref{eq:empirical-duals} of the dual, \(n\theta H^d_{a,b,B}(Q_n,y)=\sup_{\zeta\in\mathcal N_y}\sum_{t=1}^n g_t(\zeta)\), where \(g_t(\zeta):=\theta\log s_{\zeta,y}(X_t)\) is exp-concave, as noted above.
Apply the aggregation inequality for \(\Lambda=\mathcal N_y\subseteq\mathbb R^4\) to get
\[
        e^{\,n\theta H^d_{a,b,B}(Q_n,y)}
        \le
        e\,(n+1)^4\,\E_{\zeta\sim\delta_{\mathcal N_y}}e^{\sum_{t=1}^n g_t(\zeta)} .
\]
For each fixed \(\zeta\) the exponential of the sum factors as a product, \(e^{\sum_{t=1}^n g_t(\zeta)}=\prod_{t=1}^n s_{\zeta,y}(X_t)^\theta\).  Taking \(\E_P\) on both sides and interchanging it with the average over \(\zeta\) by the Fubini--Tonelli theorem gives
\[
        \E_P\,e^{\,n\theta H^d_{a,b,B}(Q_n,y)}
        \le
        e\,(n+1)^4\,\E_{\zeta\sim\delta_{\mathcal N_y}}\E_P\prod_{t=1}^n s_{\zeta,y}(X_t)^\theta .
\]

Finally, since \(X_1,\dots,X_n\) are i.i.d.\ under \(P\), the inner expectation factorizes for each fixed \(\zeta\), \(\E_P\prod_{t=1}^n s_{\zeta,y}(X_t)^\theta=(\E_P\, s_{\zeta,y}(X)^\theta)^n\), whence
\[
        \E_P\,e^{\,n\theta H^d_{a,b,B}(Q_n,y)}
        \le
        e\,(n+1)^4\,\E_{\zeta\sim\delta_{\mathcal N_y}}\big(\E_P\, s_{\zeta,y}(X)^\theta\big)^n .
\]
It remains to bound the per-\(\zeta\) moment.  For \(\zeta\in\mathcal N_y\),
\[
        \E_P\, s_{\zeta,y}(X)
        =\lambda+\alpha\E_Pe^{aX}+\beta\E_Pe^{bX}-\eta\E_Pe^{yX}
        \le\lambda+\alpha+\beta B-\eta=1,
\]
using \(\E_Pe^{aX}\le1\), \(\E_Pe^{bX}\le B\), and \(y>\gamma(P)\Rightarrow\E_Pe^{yX}\ge1\); Jensen's inequality (\(\theta<1\)) then gives \(\E_P s_{\zeta,y}(X)^\theta\le(\E_P s_{\zeta,y}(X))^\theta\le1\).  Combining the displays,
\[
        \E_P\,e^{-n\gamma f_{\rm cd}(c,Q_n)}\le e\,(n+1)^4\,e^{-n\theta c}.
\]
The polynomial prefactor is \(e^{o(n)}\), so \(\limsup_{n\to\infty}\frac1n\log\Prob_P\{M>nf_{\rm cd}(c,Q_n)\}\le-\theta c\); letting \(y\downarrow\gamma\), so that \(\theta\uparrow1\), gives the bound \(-c\).

\emph{Case \(\gamma\ge b\).}  Here \(f_{\rm cd}(c,Q)\ge(c-H^d_b(Q))_+/b\) and \(\gamma/b\ge1\), so \(e^{-n\gamma f_{\rm cd}(c,Q_n)}\le e^{-nc}e^{\,nH^d_b(Q_n)}\).  The same argument with \(g_t(\xi)=\log s^b_\xi(X_t)\), \(\Lambda=\mathcal N\subseteq\mathbb R^2\) (so \(d=2\)), and
\[
        \E_P\,s^b_\xi(X)=\lambda+\beta\E_Pe^{bX}\le\lambda+\beta=1
        \qquad(\gamma(P)\ge b\Rightarrow\E_Pe^{bX}\le1)
\]
yields \(\E_P\,e^{-n\gamma f_{\rm cd}(c,Q_n)}\le e\,(n+1)^2\,e^{-nc}\), whence \(\limsup_{n\to\infty}\frac1n\log\Prob_P\{M>nf_{\rm cd}(c,Q_n)\}\le-c\).
\end{proof}

\begin{remark}[Tractability of the conservative dual]\label{rem:tractability}
Computing \(f_{\rm cd}(c,Q)\) reduces to a one-dimensional search over \(y\in(a,b)\) of the two dual values \(H^d_{a,b,B}(Q,y)\) and \(H^d_b(Q)\).  Each is a \emph{concave} maximization (\(\zeta\mapsto\int\log s_{\zeta,y}\,dQ\) over \(\mathcal N_y\subseteq\mathbb R^4\) and \(\xi\mapsto\int\log s^b_\xi\,dQ\) over \(\mathcal N\subseteq\mathbb R^2\)), the objectives being concave because \(s_{\zeta,y}\) and \(s^b_\xi\) are affine in the dual variables and the logarithm a concave function.  The endpoint set \(\mathcal N\) is a simplex, so the only nontrivial feature is the semi-infinite positivity constraint \(s_{\zeta,y}(x)\ge0\) for all \(x\in\mathbb R\) that defines \(\mathcal N_y\).

This constraint carries an efficient {separation oracle}.  Substituting \(u=e^x\) and dividing by \(u^y>0\), positivity is equivalent to
\[
        \eta\le\inf_{u>0}\big\{\lambda u^{-y}+\alpha u^{a-y}+\beta u^{b-y}\big\},
\]
a scalar minimization whose minimizer \(u_\star\), at interior parameter values, solves the single equation \((b-y)\beta u^b=y\lambda+(y-a)\alpha u^a\).  One root-finding therefore decides membership: if the inequality holds then \(\zeta\in\mathcal N_y\); if it fails, then \(s_{\zeta,y}(x_\star)<0\) at \(x_\star=\log u_\star\), and since \(\zeta'\mapsto s_{\zeta',y}(x_\star)\) is affine, the halfspace \(\{\zeta':s_{\zeta',y}(x_\star)\ge0\}\), which contains \(\mathcal N_y\), separates \(\zeta\) from it.
\end{remark}

\section{Numerical illustration}
\label{sec:numerical-illustration}

We include a small experiment to illustrate how the different capital-buffer rules behave at moderate sample sizes.
The benchmark distribution is
$P=N(-1,1)$
so that the true adjustment coefficient is $\gamma(P)=2$.
We compare three data-driven capital coefficients.

\paragraph{Empirical plug-in.} Here the statistic is the empirical distribution, \(Q_n=P_n\), and the coefficient is the plug-in coefficient of Section~\ref{sec:plugin},
\[
        f_{\rm emp}(c,Q_n):=\frac{c}{\gamma(P_n)},
\]
with the conventions adopted there.

\paragraph{Gaussian LDP.} Here the statistic is the pair \(Q_n=(\bar X_n,V_n)\) of sample mean and variance, and the coefficient \(f_{\rm Gaus}(c,Q_n)\) is obtained from the explicit Gaussian profile in Section~\ref{subsec:gaussian}.

\paragraph{Nonparametric conservative dual.} Here the statistic is again the empirical distribution, \(Q_n=P_n\), and the coefficient \(f_{\rm Nonp}(c,Q_n)=f_{\rm cd}(c,Q_n)\) is the rule of Section~\ref{subsec:two-level-envelope}. For it we use the two-level exponential envelope
\[
        a=1,\qquad b=4,\qquad B=2e^4,
\]
which is calibrated conservatively relative to the benchmark distribution.
In the implementation, the supremum over \(y\in(a,b)\) is evaluated on a uniform discretization with \(70\) points, and each inner concave maximization is solved with the separation oracle of Remark~\ref{rem:tractability}.

For a rule \(\rho\in\{\mathrm{emp},\mathrm{Gaus},\mathrm{Nonp}\}\) with capital coefficient \(f_\rho(c,Q_n)\),
we evaluate the expected Cram\'er--Lundberg upper bound
\[
        \mathcal B^\rho_{n,c}
        :=
        \E_P\exp\{-\gamma(P)n f_\rho(c,Q_n)\}.
\]
The reported performance measure is
\[
        -\frac{\log \mathcal B^\rho_{n,c}}{cn}.
\]
The target value is one. Values below one indicate that the corresponding rule does not attain the prescribed exponential safety level on this Cram\'er--Lundberg scale; values above one indicate conservatism.

In the experiment we take \(n=50\). The empirical plug-in and Gaussian LDP rules are estimated using \(50{,}000\) Monte Carlo replications per run. The nonparametric conservative dual rule is more expensive and is estimated using \(500\) Monte Carlo replications per run. We repeat each experiment three times independently and report the pointwise median over the three runs.

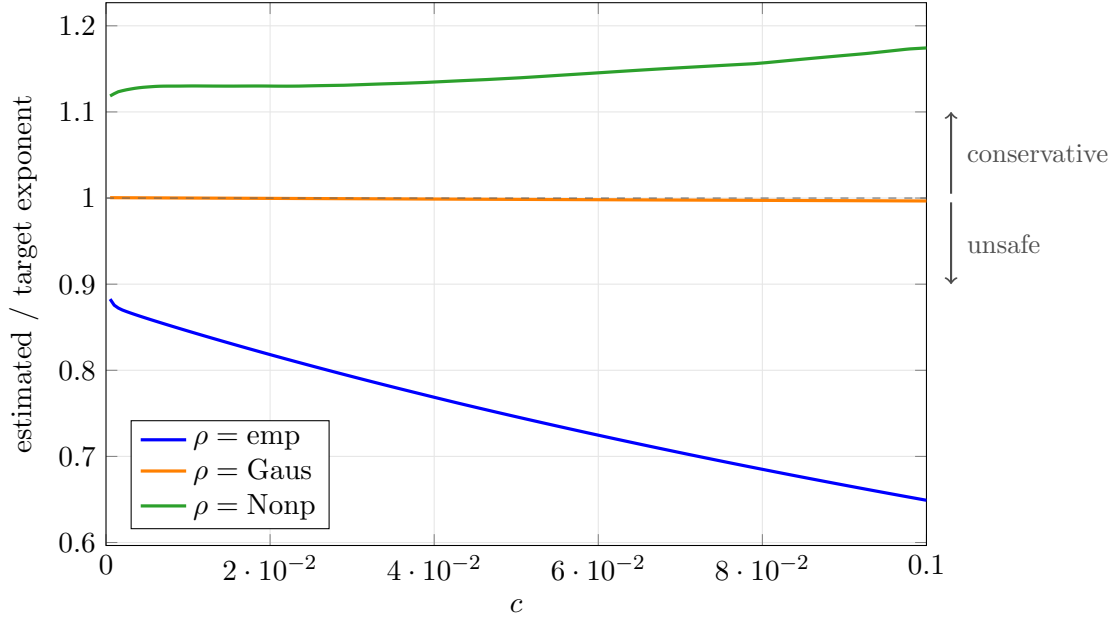
\begin{figure}[h]
        \centering
        \begin{tikzpicture}
        \definecolor{mplgreen}{HTML}{2CA02C}
        \begin{axis}[
                width=0.78\textwidth,
                height=0.55\textwidth,
                xlabel={$c$},
                ylabel={estimated / target exponent},
                xmin=0, xmax=0.1,
                xtick={0,0.02,0.04,0.06,0.08,0.1},
                legend pos=south west,
                legend cell align=left,
                grid=major,
                grid style={gray!20},
                clip=false,
        ]
        \addplot[blue,   line width=1.2pt] table[col sep=comma, x=c, y=emp]  {data/normalized_exponent_median.csv};
        \addlegendentry{$\rho=\mathrm{emp}$}
        \addplot[orange, line width=1.2pt] table[col sep=comma, x=c, y=gaus] {data/normalized_exponent_median.csv};
        \addlegendentry{$\rho=\mathrm{Gaus}$}
        \addplot[mplgreen,  line width=1.2pt] table[col sep=comma, x=c, y=nonp] {data/normalized_exponent_median.csv};
        \addlegendentry{$\rho=\mathrm{Nonp}$}
        \addplot[black!60, dashed, samples=2, domain=0:0.1, forget plot] {1};
        \draw[->, thick, black!70] (axis cs:0.103,1.005) -- (axis cs:0.103,1.10)
                node[midway, anchor=west, xshift=2pt]{\small conservative};
        \draw[->, thick, black!70] (axis cs:0.103,0.995) -- (axis cs:0.103,0.90)
                node[midway, anchor=west, xshift=2pt]{\small unsafe};
        \end{axis}
        \end{tikzpicture}
        \caption{
        Normalized finite-sample exponent
        \(-\log \mathcal B^\rho_{n,c}/(cn)\) for the empirical plug-in rule, the
        Gaussian LDP rule, and the nonparametric conservative dual rule.
        The benchmark distribution is \(P=N(-1,1)\), and \(n=50\).
        The dashed horizontal line marks the target value one.
        }
        \label{fig:numerical-normalized-exponent}
\end{figure}

Figure~\ref{fig:numerical-normalized-exponent} shows the expected qualitative ordering. The empirical plug-in rule is optimistic and falls below the target value. The Gaussian LDP rule is close to the target, reflecting that the data-generating distribution lies in the Gaussian model class used by the rule. The nonparametric conservative dual rule lies above the target. This is consistent with the theory: by strong duality the nonparametric rule evaluates the exact profile of the two-level envelope at the empirical law, and its conservatism reflects protection against the larger model class \(\mathcal P_{a,b,B}\).

\section{Conclusion}

We have developed a decision-centric large-deviations framework for choosing data-driven capital buffers in ruin models with an unknown increment law. Its central object is the profile \(f^*\), which couples the statistical cost of observing an atypical statistic with the ruin decay rate that the induced decision must control.
Within this framework we showed that the naive plug-in rule
fails to attain the prescribed safe logarithmic decay exponent, that the lower semicontinuous majorants of \(f^*\) are safe while regular rules falling strictly below \(f^*\) are, under mild conditions, unsafe, and that \(f^*\) admits explicit closed forms in three parametric models. For nonparametric classes we identified uniform integrability as the crux of well-posedness: a single exponential envelope degenerates, whereas a two-level envelope yields a genuinely data-dependent conservative dual rule \(f_{\rm cd}\) that we proved safe and, by strong duality, exact at empirical laws. A numerical illustration confirmed the expected ordering of the plug-in, Gaussian, and nonparametric rules relative to the target safety exponent.

\section*{Acknowledgments}

Bart P.G.\ van Parys gratefully acknowledges funding from NWO Vidi grant
VI.Vidi.243.021. The authors used OpenAI's ChatGPT and Anthropic's Claude for assistance 
with computations and copyediting.

\clearpage

\bibliographystyle{plain}
\bibliography{main}

\clearpage

\appendix

\section{Proofs of the parametric example profiles}\label{app:examples}

\begin{proof}[Proof of Proposition~\ref{prop:rademacher}]
Assume first \(q\in\mathcal Q_s=[0,1/2)\), the stable region.  For \(y>0\), set \(H_q(y)=\inf\{d(q\|p):0<p<1/2,\ \gamma(p)\le y\}\).  Since \(\gamma(p)\le y\) is equivalent to \(p\ge p_y:=1/(1+e^y)\), and since \(p\mapsto d(q\|p)\) is minimized at \(p=q\) (with \(\gamma(0)=+\infty\) when \(q=0\)),
\[
        H_q(y)=
        \begin{cases}
        d(q\|p_y),& 0<y<\gamma(q),\\[3pt]
        0,& y\ge \gamma(q).
        \end{cases}
\]
Since \(H_q(y)\) is precisely the inner infimum in the definition of \(c_{\rm crit}(q)\), continuity of \(p\mapsto d(q\|p)\) gives \(c_{\rm crit}(q)=\lim_{y\downarrow0}H_q(y)=d(q\|1/2)\), as asserted.
On the active interval \(0<y<\gamma(q)\),
\[
        H_q(y)
        =q\log q+(1-q)\log(1-q)+\log(1+e^y)-(1-q)y,
\]
and hence \(H_q'(y)=q-p_y\) and \(H_q''(y)=p_y(1-p_y)>0\).  Grouping the laws in Definition~\ref{def:fstar} by \(\gamma(p)\le y\) gives
$$f^*(c,q)=\sup_{y>0}(c-H_q(y))_+/y:$$
each \(p\) with \(\gamma(p)\le y\) satisfies \((c-d(q\|p))_+/y\le(c-d(q\|p))_+/\gamma(p)\), with equality in the limit \(y\downarrow\gamma(p)\).  If \(0<c<d(q\|1/2)\), the objective is zero near \(y=0\) and its
left derivative at \(y=\gamma(q)\) equals
\(-c/\gamma(q)^2\).  Hence the endpoint \(y=\gamma(q)\) cannot
maximize the objective, and the maximizer lies in the interior of
\(\{0<y<\gamma(q):H_q(y)<c\}\); for \(q=0\), where \(\gamma(q)=+\infty\), the objective is instead bounded by \(c/y\to0\) as \(y\to\infty\), and the maximizer is again interior.

The first-order
condition is \(-yH_q'(y)=c-H_q(y)\), or equivalently \(H_q(y)-yH_q'(y)=c\).  Writing \(p=p_y\), a direct calculation yields 
$$
d(q\|p)-\gamma(p)(q-p)=\Hb(p)-\Hb(q).
$$
 Therefore the first-order condition can be recast as \(\Hb(p)=\Hb(q)+c\).  Since \(\Hb\) is strictly increasing on \((0,1/2)\), this equation has a unique
solution \(p_c(q)\in[q,1/2)\), equal to \(q\) precisely when \(c=0\). At this point the previous display simplifies to
$$
c-d(q\|p_c(q))=\gamma(p_c(q))(p_c(q)-q),$$ 
and hence \(f^*(c,q)=p_c(q)-q\). 

The case \(c=0\) gives \(p_c(q)=q\) and \(f^*(0,q)=0\).  At the boundary \(c=d(q\|1/2)=\log2-\Hb(q)\), the identity above gives, for every \(0<p<1/2\),
\[
        \frac{c-d(q\|p)}{\gamma(p)}
        =(p-q)+\frac{\log2-\Hb(p)}{\gamma(p)}
        \le(p-q)+\Big(\tfrac12-p\Big)=\tfrac12-q,
\]
where the inequality holds because \(\varphi(p):=(\tfrac12-p)\gamma(p)-(\log2-\Hb(p))\) satisfies \(\varphi(1/2)=0\) and \(\varphi'(p)=-(\tfrac12-p)/(p(1-p))<0\) on \((0,1/2)\), whence \(\varphi\ge0\) there.  Hence \(f^*(c,q)\le1/2-q\).  Conversely, letting \(c'\uparrow c\) gives \(p_{c'}(q)\uparrow1/2\), so \(f^*(c',q)=p_{c'}(q)-q\to1/2-q\), and since \(f^*(\cdot,q)\) is nondecreasing, \(f^*(c,q)=1/2-q\), the boundary value. If
\(c>d(q\|1/2)\), then values \(p<1/2\) arbitrarily close to \(1/2\) satisfy
\(d(q\|p)<c\). Along such values, \(\gamma(p)\downarrow0\), while
\(c-d(q\|p)\) stays bounded away from zero. Thus, along such values, the supremum is infinite.

Finally, suppose \(q\in\mathcal Q\setminus\mathcal Q_s=[1/2,1]\), the complement of the stable region.  Then the fitted parameter \(p=q\) is infeasible, and since \(\partial_p d(q\|p)=(p-q)/(p(1-p))<0\) for \(0<p<1/2\le q\), the map \(p\mapsto d(q\|p)\) is strictly decreasing on \((0,1/2)\).  Hence, for every \(y>0\),
\[
        H_q(y)=\inf\{d(q\|p):p_y\le p<1/2\}
        =\lim_{p\uparrow1/2}d(q\|p)=d(q\|1/2)=c_{\rm crit}(q),
\]
a constant in \(y\).  Therefore \(f^*(c,q)=\sup_{y>0}(c-c_{\rm crit}(q))_+/y\), which equals \(0\) when \(c\le c_{\rm crit}(q)\) and, letting \(y\downarrow0\), equals \(+\infty\) when \(c>c_{\rm crit}(q)\).
\end{proof}

\begin{proof}[Proof of Theorem~\ref{thm:gaussian}]
For \(t\ge0\), define \(F_{m,v}(t)=\inf_{\mu<0,\sigma^2>0}\{I_{\mu,\sigma^2}(m,v)+\gamma(\mu,\sigma^2)t\}\).  By the threshold representation \eqref{eq:threshold-representation}, the profile is the
smallest \(t\ge0\) such that \(F_{m,v}(t)\ge c\), with value \(+\infty\) if no
such \(t\) exists.

Write \(s^2=\sigma^2\). The objective in the definition of \(F_{m,v}\) is
\[
        \frac{(m-\mu)^2-4\mu t+v}{2s^2}
        -\frac12-
        \frac12\log\frac{v}{s^2}.
\]
For fixed \(s^2\), the unconstrained minimizer over \(\mu\) is
\(\mu=m+2t\). This minimizer is feasible (\(\mu<0\)) precisely when
\(0\le t<-m/2\); for \(m\ge0\) this interval is empty. When it is nonempty, the minimized numerator is \(v-4t(m+t)\), and minimization over \(s^2\) gives
\[
        F_{m,v}(t)
        =
        \frac12\log\left(1-\frac{4t(m+t)}{v}\right),
        \qquad 0\le t<-\frac m2 .
\]
For \(t\ge -m/2\), the infimum over \(\mu<0\) is approached as
\(\mu\uparrow0\). The minimized numerator is then \(m^2+v\), and hence
\[
        F_{m,v}(t)
        =
        \frac12\log\left(1+\frac{m^2}{v}\right),
        \qquad t\ge -\frac m2 .
\]
Thus \(F_{m,v}\) is nondecreasing with maximal value \(c_{\rm crit}(m,v):=\frac12\log(1+m^2/v)\), attained for all \(t\ge -m/2\).  A finite solution exists if and only if \(c\le c_{\rm crit}(m,v)\), equivalently \(m^2\ge v(e^{2c}-1)\); otherwise \(F_{m,v}<c\) throughout and \(f^*=+\infty\).  Suppose \(c\le c_{\rm crit}(m,v)\).  Solving \(\frac12\log(1-4t(m+t)/v)=c\) on \([0,-m/2)\) gives \(t^2+mt+\frac v4(e^{2c}-1)=0\), whose relevant root is
\[
        t_\star=\frac{-m-\sqrt{m^2-v(e^{2c}-1)}}{2}.
\]
For \(m<0\) we have \(t_\star\in[0,-m/2]\), the minimal safe coefficient.  For \(m\ge0\) the interval \([0,-m/2)\) is empty, so \(F_{m,v}\equiv c_{\rm crit}(m,v)\ge c\) already at \(t=0\) and the minimal solution is \(t=0\); correspondingly \(t_\star\le0\).  In both cases \(f^*(c,m,v)=(t_\star)_+\) when \(m^2\ge v(e^{2c}-1)\), and \(+\infty\) otherwise.
\end{proof}

\begin{proof}[Proof of Theorem~\ref{thm:expdiff}]
For \(t\ge0\), define \(F_{u,v}(t)=\inf_{\mu>\lambda>0}\{I_{\mu,\lambda}(u,v)+(\mu-\lambda)t\}\).  By the threshold representation \eqref{eq:threshold-representation}, the profile is the
smallest \(t\ge0\) such that \(F_{u,v}(t)\ge c\), with value \(+\infty\) if no
such \(t\) exists.

The objective in the definition of \(F_{u,v}\) is
\[
        \mu(u+t)-1-\log(\mu u)
        +
        \lambda(v-t)-1-\log(\lambda v).
\]
If \(0\le t<(v-u)/2\) (an empty range when \(u\ge v\)), the unconstrained minimizers are \(\mu=1/(u+t)\) and \(\lambda=1/(v-t)\).  They satisfy \(\mu>\lambda\), and substitution gives
\[
        F_{u,v}(t)
        =
        \log\frac{u+t}{u}
        +
        \log\frac{v-t}{v}
        =
        \log\frac{(u+t)(v-t)}{uv}.
\]

It remains to treat \(t\ge(v-u)/2\). Write \(\mu=r+d\), \(\lambda=r-d\), with \(r>d>0\).  Then the objective is
\[
        r(u+v)+d(u-v+2t)-2
        -\log\bigl((r^2-d^2)uv\bigr).
\]
Since \(u-v+2t\ge0\) and \(r^2-d^2<r^2\), the objective is bounded from below by \(r(u+v)-2-\log(r^2uv)\).  This lower bound is approached as \(d\downarrow0\). Minimizing it over \(r>0\)
gives \(r=2/(u+v)\), and hence
\[
        F_{u,v}(t)
        =
        \log\frac{(u+v)^2}{4uv},
        \qquad t\ge\frac{v-u}{2}.
\]
Thus \(F_{u,v}\) is nondecreasing with maximal value \(c_{\rm crit}(u,v):=\log((u+v)^2/(4uv))\), attained for all \(t\ge(v-u)/2\).  A finite solution exists if and only if \(c\le c_{\rm crit}(u,v)\), equivalently \((v-u)^2\ge4uv(e^c-1)\); otherwise \(F_{u,v}<c\) throughout and \(f^*=+\infty\).  Suppose \(c\le c_{\rm crit}(u,v)\).  Solving \(\log((u+t)(v-t)/(uv))=c\) on \([0,(v-u)/2)\) gives \(t^2-(v-u)t+uv(e^c-1)=0\), whose relevant root is
\[
        t_\star=\frac{v-u-\sqrt{(v-u)^2-4uv(e^c-1)}}{2}.
\]
For \(u<v\) we have \(t_\star\in[0,(v-u)/2]\), the minimal safe coefficient.  For \(u\ge v\) the interval \([0,(v-u)/2)\) is empty, so \(F_{u,v}\equiv c_{\rm crit}(u,v)\ge c\) already at \(t=0\) and the minimal solution is \(t=0\); correspondingly \(t_\star\le0\).  In both cases \(f^*(c,u,v)=(t_\star)_+\) when \((v-u)^2\ge4uv(e^c-1)\), and \(+\infty\) otherwise.
\end{proof}

\end{document}